\documentclass[11pt,a4paper]{article}
\usepackage{amsfonts,amsmath,amssymb}
\usepackage{hyperref}
\usepackage{latexsym,fullpage,amsfonts,amssymb,amsmath,amscd,graphics,epic}
\usepackage[all]{xy}
\usepackage{amssymb,amsbsy,amsthm,amsxtra}
\usepackage[usenames]{color}
\usepackage{amscd}
\usepackage{amsthm}
\usepackage{amsfonts}
\usepackage{amssymb}
\usepackage{mathrsfs}
\usepackage{url}
\usepackage{bbm}
\usepackage{wasysym}
\usepackage{enumitem}
\usepackage{framed}
\usepackage{graphicx}

\theoremstyle{plain}
\newtheorem*{theorem*}{Theorem}
\newtheorem*{remark*}{Remark}
\newtheorem*{example*}{Example}
\newtheorem{lemma}{Lemma}[subsection]
\newtheorem{proposition}[lemma]{Proposition}
\newtheorem{corollary}[lemma]{Corollary}
\newtheorem{theorem}[lemma]{Theorem}
\newtheorem{conjecture}{Conjecture}
\newtheorem*{conjecture*}{Conjecture}

\newtheorem{sublemma}[lemma]{Sublemma}

\newtheorem{introtheorem}{Theorem}

\theoremstyle{definition}
\newtheorem{definition}[lemma]{Definition}
\newtheorem*{definition*}{Definition}

\newtheorem{example}[lemma]{Example}

\theoremstyle{remark}
\newtheorem{remark}[lemma]{Remark}

\newtheorem{notation}[lemma]{Notation}

\newtheorem{introremark}[introtheorem]{Remark}

\newcommand{\Hom}{\operatorname{Hom}}

\newcommand{\triv}{{\mathbbm 1}}

\newcommand{\Ind}{\operatorname{Ind}}
\newcommand{\id}{\operatorname{Id}}

\renewcommand{\Im}{\operatorname{Im}}
\newcommand{\Ker}{\operatorname{Ker}}
\newcommand{\Ext}{\operatorname{Ext}}

\newcommand{\End}{\operatorname{End}}

\newcommand{\fk}{{\mathbbm{k}}}
\newcommand{\Z}{{\mathbb Z}}

\newcommand{\eps}{{\varepsilon}}

\newcommand{\gl}{\mathfrak{gl}}
\newcommand{\ad}{\operatorname{ad}}

\newcommand{\abs}[1]{\left|{#1}\right|}

\newcommand{\g}{\mathfrak{g}}

\newcommand{\T}{\mathcal{T}}

\newcommand{\cU}{\mathcal{U}}

\newcommand{\G}{\mathbb{G}_a}

\newcommand{\Res}{\mathrm{Res}}

\newcommand{\Vect}{\mathtt{Vect}}
\newcommand{\sVect}{\mathtt{sVect}}
\newcommand{\Rep}{\mathrm{Rep}}

\newcommand{\osp}{\mathfrak{osp}}

\newcommand{\comment}[1]

\newcommand{\InnaA}[1]{{#1}} 
 
\newcommand{\InnaC}[1]{{#1}}
\newcommand{\InnaD}[1]{{#1}}
\newcommand{\VeraA}[1]{{#1}}
\def\quotient#1#2{%
    \raise1ex\hbox{$#1$}\Big/\lower1ex\hbox{$#2$}%
}

\begin{document}

\date{\today}
\title{Jacobson-Morozov Lemma for Algebraic Supergroups}
 \author{Inna Entova-Aizenbud\footnote{Inna Entova-Aizenbud, Dept. of Mathematics, Ben Gurion University,
Beer-Sheva,
Israel; email: \href{mailto:entova@bgu.ac.il}{entova@bgu.ac.il}.}, Vera Serganova\footnote{Vera Serganova, Dept. of Mathematics, University of California at
Berkeley,
Berkeley, CA 94720; email: \href{mailto:serganov@math.berkeley.edu}{serganov@math.berkeley.edu}.}}

\maketitle
\setcounter{tocdepth}{3}
\begin{abstract}
 
 Given a quasi-reductive algebraic supergroup $G$, we use the theory of 
semisimplifications of symmetric monoidal categories to define a symmetric 
monoidal functor $$\Phi_x: \Rep(G) \to \Rep(OSp(1|2))$$ associated to any given 
element $x \in \mathrm{Lie}(G)_{\bar 1}$. For nilpotent elements $x$, we show 
that the functor $\Phi_x$ can be defined using the Deligne filtration associated 
to $x$. 
 
 We use this approach to prove an analogue of the Jacobson-Morozov Lemma for algebraic supergroups. Namely, we give a necessary and sufficient condition on odd nilpotent elements $x\in \mathrm{Lie}(G)_{\bar 1}$ which define an embedding of supergroups $OSp(1|2)\to G$ so that $x$ lies in the image of the corresponding Lie algebra homomorphism. 
\end{abstract}
\setcounter{tocdepth}{3}
\section{Introduction}\label{sec:intro}

\subsection{}

Let $\fk$ be an algebraically closed field of characteristic zero.

In the classical setting, one has the following version of the Jacobson-Morozov lemma: given an embedding of algebraic groups
$\G \to L$, where $L$ is reductive and $\G$ is the additive group, one can extend this homomorphism to a homomorphism $SL_2 \to L$.

The latter homomorphism is unique up to conjugation by an element of $L$. 

The embedding $\G \to L$ of course corresponds to a choice of nilpotent element in $\mathrm{Lie}(L)$. This provides the Jacobson-Morozov lemma for Lie algebras: every nilpotent element in a semisimple Lie algebra $\mathfrak{l}$ can be embedded into an
$\mathfrak{sl}_2$-subalgebra of $\mathfrak{l}$, and this embedding is unique up to conjugation by an element of $L$. 

In this paper, we extend this result to the case of a quasi-reductive algebraic supergroup $G$ (here "quasi-reductive" means that the even part $G_{\bar 0}$ of the supergroup $G$ is a reductive algebraic group).

The additive group $\G$ will then be replaced by a supergroup $\G^{(1|1)}$ whose Lie superalgebra $\g_a^{(1|1)}$ is a nilpotent Lie superalgebra, spanned by an odd nilpotent element $x$ and its commutator $[x,x]$. 

Fix an algebraic supergroup $G$ with corresponding Lie superalgebra $\g = \mathrm{Lie}(G)$ and an odd nilpotent element $y \in \g_{\bar 1}$. We then have an homomorphism $\G^{(1|1)} \to G$ of algebraic supergroups, whose differential $\g_a^{(1|1)} \to \g_{\bar 1}$ sends $x$ to $y$.

The group $SL_2$ will be replaced by the supergroup\footnote {\InnaC{We apologize for assuming connectedness of $OSp(1|2)$ to avoid the 
awkward notation $SOSp(1|2)$.}} $OSp(1|2)$, the latter being one of a few algebraic supergroups whose
category of finite-dimensional representations is semisimple. The group $\G^{(1|1)}$ embeds into $OSp(1|2)$, and $\mathfrak g_a^{(1|1)}$ is isomorphic to a maximal nilpotent subalgebra of $\mathfrak{osp}(1|2)$.



One should state right away that the situation here is trickier than in the classical setting: we do not expect an embedding $\G^{(1|1)} \to G$ (corresponding to a choice of an odd nilpotent element  in $\g$) to necessarily give an embedding $OSp(1|2) \to G$.
Indeed, the irreducible finite-dimensional representations of $OSp(1|2)$ have 
categorical dimension (also called ``superdimension'') $\pm 1$,  therefore we 
must have that the restriction of \InnaC{every} finite-dimensional $G$-module to 
$\G^{(1|1)}$ is a 
direct sum of indecomposable $\G^{(1|1)}$-modules of categorical dimension $\pm 1$. 

To state our main result, we will use the following definition:

\begin{definition*}
Let $V$ be a finite-dimensional vector superspace with an action of $\G^{(1|1)}$ on it.
  
The action is called {\it neat} if as a $\G^{(1|1)}$-module, all the indecomposable summands of $V$ have non-zero categorical dimension (``superdimension'').
\end{definition*}
The action of $\G^{(1|1)}$ on $V$ is completely determined by the odd nilpotent operator $x \in \End(V)$. The operator $x$ is called {\it neat} if it defines a neat action of $\G^{(1|1)}$ on $V$. 

Our main result is the following ``odd'' version of Jacobson-Morozov Lemma in the superalgebra setting (see Theorem \ref{thrm:JM}):
\begin{introtheorem}\label{introthrm:JM}
 
 Let $G$ be a quasi-reductive algebraic supergroup, and $\g=Lie(G)$ its Lie 
superalgebra. Let $x \in \g_{\bar 1}, x\neq 0$ be a nilpotent element such that 
$x\rvert_V$ is neat, for \InnaC{every} finite dimensional (algebraic) 
representation $V$ of $G$. 
 
 Let $i:\G^{(1|1)} \hookrightarrow G$ be the homomorphism of algebraic supergroups corresponding to the inclusion $x \in \g$.
 
 Then the inclusion $i$ can be extended to an injective homomorphism $\bar{i}: OSp(1|2) \hookrightarrow G$.
\end{introtheorem}

Moreover, we give a simple criterion to check that an element $x \in \g_{\bar{1}}$ is neat: given a quasi-reductive algebraic supergroup $G$, a faithful representation $V$ of $G$, and $x\in \g_{\bar{1}} \setminus\{0\}$, we prove that the operator $x\rvert_V$ is nilpotent and neat iff $x$ satisfies the conditions of Theorem \ref{introthrm:JM}.

\begin{introremark}
{As it was stated above, the neatness of $x$ is also a necessary condition in order for the homomorphism $\bar{i}$ extending $i$ to exist.}
\end{introremark}

Let us give a short overview of our main tool for proving Theorem \ref{introthrm:JM}. 

Consider the category $\Rep(\G^{(1|1)})$. Using the theory of semisimplification, we show that there exists a (non-exact) $\fk$-linear full symmetric monoidal functor $$S: \Rep(\G^{(1|1)}) \to \Rep(OSp(1|2))$$ making $\Rep(OSp(1|2))$ the universal semisimple quotient of $\Rep(\G^{(1|1)})$. The functor $S$ annihilates all indecomposable $\G^{(1|1)}$-representations of superdimension zero.

Now, let us go back to the setting of Theorem \ref{introthrm:JM}. Consider the restriction functor $R_x: \Rep(G) \to \Rep(\G^{(1|1)})$. Setting $$\Phi_x := S \circ R_x: \,  \Rep(G) \longrightarrow \Rep(OSp(1|2)),$$ we show that this functor is an exact symmetric monoidal functor between super-Tannakian categories. Therefore it defines a homomorphism $\bar{i}$ as in Theorem \ref{introthrm:JM}.

Another construction of $\Phi_x$ is as follows. Let $M \in \Rep(G)$, and consider the action of $x$ on it. This action defines a canonical finite increasing ``Deligne'' filtration
$$\ldots\subset\mathcal F^{i}(M)\subset \mathcal F^{i+1}(M)\subset\ldots $$ satisfying conditions similar to that of the classical Deligne filtration appearing in the Hodge theory.

Let $Gr^{i}(M) = \mathcal F^{i}(M) / \mathcal F^{i-1}(M)$. Since $x$ is neat, we have: $Gr^{2i+1} (M) =0$ for all $i$. Then the grading and the action of $x$ extends uniquely to an  action of $OSp(1|2)$ on $\bigoplus_i Gr^{2i} (M)$ and an isomorphism of
$OSp(1|2)$-modules $$\bigoplus_i Gr^{2i} (M) \cong \Phi_x(M).$$

The nilpotent elements in $\g_{\bar 1}$ satisfying the condition on Theorem \ref{introthrm:JM} are called {\it neat} elements, and the set of such elements is denoted $\g_{neat}$.

We initiate the study of the set $\g_{neat}$; it is stable under the adjoint action of $G_{\bar 0}$ and if $G$ is quasi-reductive, we show that $\g_{neat}$ has finitely many $G_{\bar 0}$-orbits, and $\g_{\bar 1} = \g_{neat}$ iff $\Rep(G)$ is semisimple.

The construction of the functor $\Phi_x$ using semisimplification is inspired by \cite{EO}, and can be extended to arbitrary odd elements in the Lie superalgebra $\g$. Each element $x\in \g_{\bar 1}$ defines a functor $\Phi_x : \,  \Rep(G) \longrightarrow \Rep(OSp(1|2))$ in a similar manner.

This allows one to define the notion of {\it support} of a module $M \in \Rep(G)$ as the subset $${supp(M):=\{x \in \g_{\bar 1} | \Phi_x(M)\neq 0\}.}$$ We study supports of modules in Section \ref{sec:tensor_func}. We show that the minimal support a module can have is $\g_{neat}$, and this occurs when $M$ is projective. 

\subsection{Acknowledgement}
We would like to thank Joseph Bernstein, Pavel Etingof and Victor Ostrik for helpful discussions, \InnaC{and Kevin Coulembier for helpful remarks}. The authors were supported by the NSF-BSF grant NSFMath 2019694.

\section{Notation}\label{sec:notation}

Our base field will be an algebraically closed field $\fk$ with $char(\fk) = 0$.
\subsection{Tensor categories and vector superspaces}\label{ssec:symmetric monoidal_cat}

All our categories will be $\fk$-linear rigid symmetric monoidal, with the bifunctor $-\otimes-$ being bilinear. All the functors will be symmetric monoidal and $\fk$-linear.

We will write SM for short when refering to {\it symmetric monoidal} (both categories and functors).

Throughout the paper, we will use the following terminology and assumptions, following \cite{EGNO}:

\begin{itemize}
 \item 
 A {\it tensor} category is an abelian rigid SM $\fk$-linear 
category, where $-\otimes-$ is biexact\footnote{In fact, this follows from bilinearity of $-\otimes-$.}. A {\it tensor} functor between tensor categories is an exact SM functor.

\item We will assume that $\End(\triv) = \fk$ in all our categories.

\item In a rigid SM category $\cU$, one defines the {\it trace} of $f \in \End(C)$, $C \in \cU$ as $tr(f) \in \End(\triv)$ where
$$tr(F): \triv \to C \otimes C^* \xrightarrow{f \otimes \id_{C^*}} C\otimes C^* 
\xrightarrow{b_{C, C^*}} C^* \otimes C \to \triv.$$
 The (categorical) {\it dimension} of an object $C$ is then defined as $\dim(C) =tr(\id_C)
\in \End(\triv)$.

A morphism $f:C_1 \to C_2$ in $\cU$ is called {\it negligible} if it satisfies the following condition:
$$ \forall g: C_2 \to C_1, \; tr(g \circ f) = 0.$$

The set $\mathcal{N}$ of all negligible morphisms in a rigid SM $\fk$-linear $\cU$ forms an ideal under composition and tensor product, hence $\cU / \mathcal{N}$ is again a rigid SM $\fk$-linear category.
\item {The} {\it semisimplification} of a rigid SM $\fk$-linear category $\cU$ is the pair $(S, \overline{\cU})$, where $\overline{\cU} = \cU /\mathcal{N} $, and $S: \cU \to \overline{\cU}$ is the quotient functor. 

One can immediately see that
$\overline{\cU}$ is a semisimple rigid SM $\fk$-linear category, and $S: \cU \to \overline{\cU}$ is a full SM $\fk$-linear functor. 
In fact, the pair $(S, \overline{\cU})$ is universal among pairs, cf. for example \cite{AndreKahn, EO}, and also \cite{Heid, BEEO}.
\end{itemize}

Two {important} examples of rigid SM categories are the 
category of 
finite-dimensional representations of a group $G$ and the category of 
finite-dimensional vector superspaces $\sVect$, defined below.

\subsection{Vector superspaces and supergroups}\label{ssec:supervec}

A {\it 
vector superspace} is a $\mathbb{Z}/2\mathbb{Z}$-graded $\fk$-vector space 
$V=V_{\bar 0}\oplus V_{\bar 1}$; for $v\in V_{\eps}$, $\eps\in  
\mathbb{Z}/2\mathbb{Z}$, we denote by $\bar{v} = \eps$ the {\it parity} of $v$. 
{Given two superspaces $V, W$, the space $\Hom_{\fk}(V, W)$ is} {naturally} 
$\Z/2\Z$-graded {as well}, with 
homogeneous morphisms called {\it even} and {\it odd} respectively. 

{The objects in the category of vector superspaces $\sVect$ are 
finite-dimensional vector superspaces and the morphisms are linear 
even morphisms: $$\Hom_{\sVect}(V, W) = \Hom_{\fk}(V, W)_{\bar 0}.$$} 

The category $\sVect$ has a monoidal structure 
given by $(\otimes, \fk^{1|0})$, {with {the} symmetry morphisms}
$$b_{V, W}:V \otimes W \to W \otimes V,\;\; v\otimes w  \mapsto 
(-1)^{\bar{v}\bar{w}} w \otimes v$$
This makes $\sVect$ a rigid SM
category, which is {\bf not} equivalent to the {SM} category 
$\Rep(\Z/2\Z)$. 

The (categorical) dimension of $V \in \sVect$, also called 
{\it superdimension}, is then 
$\dim V = \dim V_{\bar{0}} - 
\dim V_{\bar 1}$. {Sometimes we also denote the dimension of $V$ as a vector space by $(\dim V_{\bar{0}} | \dim V_{\bar{1}})$.

We will denote by $\Pi$ the change of parity endofunctor on $\sVect$: namely, $\Pi \fk^{m|n} \cong \fk^{n|m}$.}

In the category $\sVect$ one can define Lie algebra objects, called (finite-dimensional) {\it Lie 
superalgebras}. An 
example of such an object is the vector superspace $\End(V)$ with a (signed) commutator bracket,
denoted $\gl(V)$. For the vector superspace $V=\fk^{m|n}$, we denote $\gl(m|n)=\gl(V)$. 

Similarly, one can consider {\it algebraic (affine) supergroups}\footnote{By ``algebraic (super)group'' we always mean an affine algebraic group (super)scheme $G$; its Hopf (super)algebra $O(G)$ is finitely generated.}. These form a category which is opposite to the category of finitely-generated commutative Hopf algebra ind-objects in $\sVect$, and each algebraic supergroup $G$ has a (finite-dimensional) Lie superalgebra $\mathrm{Lie}(G)$ attached to it.

A {\it pro-supergroup} $G$ is a limit of supergroups; its algebra of functions $O(G)$ is a Hopf superalgebra, not necessarily finitely generated.

\begin{definition}
 Given a Lie superalgebra $\g$, the category of finite-dimensional representations of $\g$ is denoted by $\Rep(\g)$ has objects $(V, \rho)$ where $V \in \sVect$ and $\rho: \g \to \gl(V)$ is a homomorphism of Lie algebra objects in $\sVect$. The maps in $\Rep(\g)$ would be $\g$-equivariant {(even)} maps of vector superspaces.
\end{definition}

\begin{definition}

Given a supergroup $G$, we define $\Rep(G)$ as the category of representations of $G$ in $\sVect$. 
\end{definition}

Let $G$ be an algebraic supergroup with Lie superalgebra $\g$ and with the underlying (affine) algebraic group $G_{\bar 0}$. A {\it $(\g,G_{\bar 0})$-module} is by definition a $\g$-module and
  $G_{\bar 0}$-module such that the differential of the $G_{\bar 0}$-action on $M$ coincides with the $\g_{\bar 0}$-action.
\begin{theorem}\label{theorem_HC} \cite{Masuoka} The category $Rep(G)$ is equivalent to the category of finite-dimensional $(G,\g_{\bar 0})$-modules. 
    \end{theorem}

\begin{remark}\label{rmk:faithful}
Any algebraic (affine) supergroup has a faithful finite-dimensional representation. This is proved in the same way as for affine algebraic groups (see e.g. \cite[Chapter 4, Par. 9]{Milne}).
\end{remark}

\begin{definition}
 A supergroup $G$ is called {\it quasi-reductive} if $G_{\bar{0}}$ is reductive\footnote{By a reductive algebraic group we mean an algebraic group whose finite-dimensional (rational) representations form a semisimple category.}. 
\end{definition}

\subsection{Tannakian categories}\label{ssec:Tannakian}
In this section we use the terminology of \cite{D}.

For a pre-Tannakian category $\T$ and a $\T$-group $G$, we have a group homomorphism $\eps:\pi(\T) \to G$, where $\pi(\T)$ denotes the fundamental group of $\T$, in the sense of \cite{D}. Define $\Rep_{\T}(G, \eps)$ to be the category of representations $M$ of $G$ in $\T$ whose composition with $\eps$ is the natural $\pi(\T)$-action on $M$, seen as an object in $\T$.

\InnaC{
Let $G$ be a $\T$-groups and let $R:\Rep_{\T}(G) \to \T$ be the forgetful functor. The image in $\T$ of the fundamental group of $\Rep_{\T}(G)$ under $R$ is then $G\rtimes\pi(\T)$ (see \cite[Appendix 2]{ES}). 

Let $G_1, G_2$ be two $\T$-groups and let $F:\Rep_{\T}(G_1) \to \Rep_{\T}(G_2) $ be an exact symmetric monoidal functor. Tannakian formalism (as described in \cite{D}) then tells us that $F$ induces a homomorphism of $\T$-groups $G_2\rtimes\pi(\T) \to G_1 \rtimes \pi(\T)$.}

If $\T = Rep(G')$, we write $\Rep_{G'}(G, \eps) = \Rep_{\T}(G, \eps)$ for short.

Deligne's theorem on super-Tannakian reconstruction states that given a finitely $\otimes$-generated pre-Tannakian category $\T$ where each object is annihilated by some Schur functor, the category $\T$ is equivalent to $\Rep_{\sVect}(G, \eps)$ for some algebraic supergroup $G$ and $\eps: \mu_2 \to G$, where $\mu_2 \cong \pi(\sVect) \cong \{\pm 1\}$ and $\eps$ is the corresponding supergroup homomorphism. \InnaC{Such categories are called {\it super-Tannakian categories}.}

\begin{remark}\label{rmk:conn_Tannakian}
 \InnaC{Let $G_1, G_2$ be supergroups, and assume $G_2$ is connected. Let $F: \Rep(G_1) \to \Rep(G_2)$ be an exact symmetric monoidal functor. Then $F$ induces a homomorphism of supergroups $f:G_2 \rtimes \mu_2 \to G_1 \rtimes \mu_2$, which restricts to a homomorphism $ f:G_2  \to G_1 \rtimes \mu_2$. But since $G_2$ is connected, $f$ is in fact a homomorphism $ f:G_2  \to G_1$.}
\end{remark}

\InnaC{The following lemma will be useful for us when considering super Tannakian categories:}
\InnaC{
  \begin{lemma}\label{lem:splittannakian} Let $\mathcal T$ be a super Tannakian category which contains $\Pi(\mathbb \triv)$ (the unit object with shifted parity). Then $\mathcal T$ is equivalent to $\Rep(G)$ for some algebraic
    pro-supergroup $G$. 
\end{lemma}
\begin{proof} We know from Tannakian formalism that $\mathcal T$ is equivalent to $\Rep(\tilde G,\eps)$, where $\tilde G$ is some supergroup, $\eps:\mu_2\to \tilde G$ is a homomorphism and the $g\in\mu_2,g\neq 1$ acts by the
  $\mathbb Z_2$-grading on objects of $\mathcal T$.
  Let $G$ be the kernel of the representation of $\tilde G$ in $\Pi(\triv)$. Then $\Rep(\tilde G,\eps)$ is equivalent to $\Rep(G)$.
  \end{proof}
}
  
   We now prove give a ``categorical characterization of quasi-reductive supergroups''.
\begin{proposition}\label{prop:enough_proj}
Let $G$ be an algebraic supergroup. Then $\Rep(G)$ has enough projectives if and only if $G$ is quasi-reductive.
\end{proposition} 
 \begin{remark}
  Clearly, the existence of enough projectives in $\Rep (G,\eps)$ is equivalent to the existence of enough projectives in $\Rep(G)$.
 \end{remark}

\begin{proof}
 
    The induction $\Ind_{G_{\bar 0}}^G$ and restriction $\Res_{G_{0}}$ define two adjoint functors between $\Rep (G)$ and $\Rep (G_{\bar{0}})$. In fact, taking into account Theorem \ref{theorem_HC}, we have the natural isomorphism
    $$\Ind_{G_{\bar 0}}^G(?)\cong \Hom_{U(\g_{\bar 0})}(U(\g),?),$$
    which explains why in this case the induction is an exact functor; it maps a finite-dimensional module to a finite-dimensional module, {and every finite-dimensional module is a submodule of $\Ind^G_{G_{\bar 0}} M$ for some
      $M \in \Rep (G_{\bar{0}})$}.
    By definition of the induction functor, if $M$ is an injective $G_{\bar{0}}$-module
    then $\Ind_{G_{\bar 0}}^GM$ is an injective $G$-module. The existence of enough projectives is equivalent to existence of enough injectives by duality. Hence if $G_{\bar 0}$ is reductive, the category $\Rep(G)$ has enough projectives.

   For converse, consider the functor $J:\Rep(G_{\bar 0})\to\Rep(G)$ defined by
  $$J(M)=U(\g)\otimes_{U(\g_{\bar 0})}M.$$
    This functor is left-adjoint to $\Res_{G_{0}}$ and isomorphic to $\Ind_{G_{\bar 0}}^G$ after some twist. We have
    $$\Hom_{G_{\bar 0}}(X, \Res_{G_{\bar 0}}Y)\cong\Hom_{G}(J(X), Y),$$
    which implies that $\Res_{G_{\bar 0}}$
    maps projectives to projectives  and hence injectives to injectives. Therefore, if $\Rep (G)$ has enough projectives, the same is true for $\Rep(G_{\bar 0})$.
If $P\in \Rep(G_{\bar 0})$ is projective, then $P\otimes P^*$ is projective.
 {If $P\neq 0$ then $\dim P\neq 0$,} and the trivial module $\fk$ splits as a direct summand in $P\otimes P^*$. Therefore, $\fk$ is projective. That means $\Ext_{G_{\bar 0}}^1(\fk,M)=0$ for any $M\in \Rep(G_{\bar 0})$ and therefore
$\Ext_{G_{\bar 0}}^1(\fk,N^*\otimes M)=\Ext_{G_{\bar 0}}^1(N, M)=0$ for any 
$N,M\in \Rep(G_{\bar 0})$. Therefore $G_{\bar 0}$ is reductive.
\end{proof}

\subsection{The Duflo-Serganova functor \texorpdfstring{$DS$}{DS}}
Let $\g$ be a Lie superalgebra, and let $x\in\g_{\bar{1}}$ be an odd element satisfying $[x,x]=0$. In \cite{DS} M. Duflo and  V. Serganova defined a functor $$DS_x: \Rep(\g) \longrightarrow \Rep(\g_x), \;\; M \longmapsto M_x:=\quotient{\Ker x\rvert_M}{\Im x \rvert_M}$$ where $\g_x:=\Ker\ad_x/\Im \ad_x$ is again a Lie superalgebra.

\begin{example}
 For $\g=\gl(m|n)$ and $x \in \g_{\bar{1}}$ of rank $1$, we have: $\g_x \cong \gl(m-1|n-1)$.
\end{example}

This functor is symmetric monoidal {(hence preserves categorical dimensions)}, but in general not exact on either side. 
\subsection{Representations of \texorpdfstring{$OSp(1|2)$}{OSp(1|2)}}\label{ssec:osp}
Recall that all isomorphism classes of indecomposable representations of the usual additive algebraic group $\G$ are enumerated by their dimensions.

Let $V_k$ be the indecomposable $(k+1)$-dimensional representation of the usual additive algebraic group $\G$ (we set $V_{-1}:={0}$). The action of $\G$ extends to an action of $SL_2$, and $V_k$ is an irreducible representation of $SL_2$.

The supergroup $OSp(1|2)$ is the group (super)scheme in $\sVect$ of 
automorphisms of the space $\fk^{1|2}$ respecting a fixed symmetric 
non-degenerate form $\fk^{1|2} \otimes \fk^{1|2} \to \fk$ and having Berezinian 
$1$ \InnaC{(so $OSp(1|2)$ is a connected group in our convention)}.

We have: $OSp(1|2)_{\bar 0} = SL_2$, and $$\osp(1|2) = \mathrm{Lie}(OSp(1|2))$$ is a $(3|2)$-dimensional Lie superalgebra with even part $\mathfrak{sl}_2$, and the odd part (as a representation of $\mathfrak{sl}_2$) isomorphic to the standard $2$-dimensional representation $V_1$. {We denote by $h$ the generator of the Cartan subalgebra in $\osp(1|2)_{\bar 1} \cong \mathfrak{sl}_2$ and by $X, Y$ the standard basis of $\osp(1|2)_{\bar 1} \cong V_1$. The elements $h, X, Y$ generate the superalgebra $\osp(1|2)$, with relations
$$[h, X]=-2X,\, [h, Y]=2Y, \, [Y, X]=h.$$ 
\comment{$$[h, X]=X, \, [h, Y]=-Y, \, [e, X] = 0,\, [e, Y]=X, \,[f, X] = Y,\, [f, Y]=0$$ and }}

The category $\Rep(OSp(1|2))$ is semisimple, with isomorphism classes of simple objects (up to parity switch) numbered by even integers: namely, we denote by 
$\widetilde{M}_{2k}$ ($k\geq 0$) the $(k+1|k)$-dimensional irreducible representation such that $$Res_{OSp(1|2)_{\bar 0}} \widetilde{M}_{2k} \cong V_{k+1} \oplus \Pi V_k$$ as $SL_2$-representations, and $\osp(1|2)_{\bar{1}}$ acts by odd morphisms accordingly.

\section{The unipotent additive supergroup}
\subsection{Definition}
Let $\G^{(1|1)}$ be the $(1|1)$-dimensional additive algebraic supergroup with 
$$\left({\G}^{(1|1)}\right)_{\bar{0}} = \G.$$ The corresponding 
$(1|1)$-dimensional Lie superalgebra has a basis $[x,x], x$, where $x \neq 0$ is 
an odd element, with relation (the Jacobi identity) $[x,[x,x]]=0$. This defines 
a Harish-Chandra pair as in \cite{Masuoka} and hence an algebraic supergroup. 

\InnaC{The ring of functions on this supergroup is the supercommutative algebra
$$ \mathcal{O}\left({\G}^{(1|1)}\right) = \quotient{\fk[t, \omega]}{\omega^2=0}$$ where $t$ is even and $\omega$ is odd. This is a Hopf algebra, the counit map given by quotienting by the ideal $<t, \omega>$, the antipode - by multiplication by $(-1)$, and the comultiplication given by $$\Delta(\omega) = 1 \otimes \omega + \omega \otimes 1, \; \Delta(t) =  1 \otimes t + t \otimes 1 +\omega \otimes \omega.$$}

Denote $\cU = \Rep(\G^{(1|1)})$.

Let $M_k$ be the indecomposable representation of $\G^{(1|1)}$ with a basis $a_0, a_1, \ldots, a_{k}$, with $\overline{a_{j}} \equiv j \mod 2$ for any $j\geq 0$, and
$$x.a_j = 
\begin{cases}
	a_{j+1} &\text{ if } j<k\\
	0 &\text{ if } j=k                                                                                                                                                                                   
\end{cases}
$$ 
It is easy to see that these are all the indecomposable representations of $\G^{(1|1)}$ up to change of parity and isomorphisms.

\begin{notation}
 By $\cU_{neat} \subset \cU$ we denote the full subcategory of objects $M \in \cU$ such that every indecomposable direct summand of $M$ has non-zero dimension. 
\end{notation}

\subsection{Relation with \texorpdfstring{$OSp(1|2)$}{OSp(1|2)}}

Let $x \in \osp(1|2)_{\bar 1}$ be such that $x^2=\frac{1}{2}[x,x] \in \osp(1|2)_{\bar 0}$ corresponds to 
$\begin{pmatrix}
0 &0\\
1 &0                                                                                                       \end{pmatrix}
$ under the isomorphism $\osp(1|2)_{\bar 0} \cong \mathfrak{sl}_2$.
The element $x$ defines an embedding $\G^{(1|1)} \hookrightarrow OSp(1|2)$, which in turn induces a tensor functor $$S^*: \Rep(OSp(1|2)) \to \cU.$$ It is easy to see that $S^*(\widetilde{M}_{2k}) \cong M_{2k}$ for any $k\geq 0$.

\subsection{Clebsh-Gordan coefficients for \texorpdfstring{$\G^{(1|1)}$}{additive (1|1)-dimensional supergroup}}

We have: $$Res^{\G^{(1|1)}}_{\G} M_{2k+1} \cong V_k \oplus \Pi V_k, \;\; Res^{\G^{(1|1)}}_{\G} M_{2k} \cong V_k \oplus \Pi V_{k-1}$$
\begin{example}
 We have: $M_0 = \triv$, $Res^{\G^{(1|1)}}_{\G} M_1 \cong \triv \oplus \Pi \triv$.
\end{example}

\begin{lemma}\label{lem:induction} We have $$M_{2k+1}\simeq \Ind^{\G^{(1|1)}}_{\G}\Pi V_k$$ or, equivalently, $M_{2k+1}\simeq U(\g^{(1|1)}_a)\otimes _{U(\g_a)}V_k$. 
\end{lemma}
\begin{remark}
{Recall that induction between algebraic groups corresponds to coinduction between Lie algebras.}
\end{remark}
\begin{proof} Immediate straightforward computations.
  \end{proof}

\begin{lemma}\label{lem:Clebsh_Gordan}
 The $\G^{(1|1)}$-module decomposition of tensor products of indecomposables into indecomposable summands is as follows:

\begin{align*}
 &M_{2k}\otimes M_{2m} \cong \bigoplus_{s=\abs{k-m}}^{k+m} \Pi^{k+m-s} M_{2s}\\
 &M_{2k+1} \otimes M_{2m}\cong \bigoplus_{s=min(\abs{k-m}, \abs{k-m+1})}^{k+m} \Pi^{k+m-s} M_{2s+1}\\
 &M_{2k+1} \otimes M_{2m+1}\cong \bigoplus_{\substack{\abs{k-m}\leq s\leq {k+m}\\ s\equiv k-m \mod 2}} M_{2s+1} \oplus \Pi M_{2s+1}
\end{align*}
\end{lemma}
\begin{proof}
The first identity follows from the Clebsh-Gordan identity for $OSp(1|2)$ using the fact that {$S^*(\widetilde{M}_{2k}) \cong M_{2k}$ for any $k\geq 0$}.
  For the second and the third identities we use that for any pair of supergroups $H\subset G$ and $M\in \Rep(H), N\in\Rep(G)$ we have a natural isomorphism
  $$ \Ind^G_H(M)\otimes N\cong\Ind^G_H(M\otimes \Res^G_HN).$$
  In particular,
  $$M_{2k+1} \otimes M_{2m}\cong\Ind^{\G^{(1|1)}}_{\G}\Pi V_k\otimes M_{2m}\cong\Ind^{\G^{(1|1)}}_{\G}(\Pi V_k\otimes (V_{m}\oplus \Pi V_{m-1})).$$
  Using the the Clebsh-Gordan coefficients for $SL(2)$ we get
  $$\Pi V_k\otimes (V_{m}\oplus \Pi V_{m-1})=\bigoplus_{\substack{\abs{k-m}\leq s\leq k+m,\\ s\equiv m-k \mod 2}}V_s\oplus\bigoplus_{\substack{\abs{k-m+1}\leq s\leq k+m-1,\\ s\equiv m-k+1 \mod 2}}\Pi V_s.$$
  Hence
  $$\Ind^{\G^{(1|1)}}_{\G}(\Pi V_k\otimes (V_{m}\oplus \Pi V_{m-1}))=\bigoplus_{s=min(\abs{k-m}, \abs{k-m+1})}^{k+m} \Pi^{k+m-s} M_{2s+1}.$$
  The proof of the third identity is similar.
  
 \comment{
 First, consider the $\G^{(1|1)}$-representation $M_{2k}\otimes M_{2m}$. Restricting to $\G$, we have:
 
\begin{align*}
 &Res^{\G^{(1|1)}}_{\G} \left( M_{2k}\otimes M_{2m}\right) \cong \left(V_k \oplus \Pi V_{k-1}\right) \otimes \left(V_m \oplus \Pi V_{m-1} \right)  \cong \\
 &\cong \bigoplus_{s=\abs{k-m}}^{k+m} V_s \oplus \bigoplus_{s=\abs{k-m}}^{k+m-2} V_s \oplus \Pi  \bigoplus_{s=\abs{k-m-1}}^{k+m-1} V_s \oplus \Pi  \bigoplus_{s=\abs{k-m+1}}^{k+m-1} V_s
 \end{align*}

Due to the decompositions of $Res^{\G^{(1|1)}}_{\G} M_{2s+1}, Res^{\G^{(1|1)}}_{\G} M_{2s}$, we conclude immediately that $$M_{2k}\otimes M_{2m} \cong \bigoplus_{s=\abs{k-m}}^{k+m} \Pi^{k+m-s} M_{2s}.$$

Now, we consider the remaining decomposition of $M_{2k+1}\otimes M_{2m}$. We have the following special case:
 \begin{sublemma}
  For any $m\geq 1$, we have: $$M_{2m} \otimes M_{1}\cong M_{1}\otimes M_{2m}\cong M_{2m+1} \oplus \Pi M_{2m-1}$$
  and we have: $$M_{2m+1} \otimes M_{1}\cong M_{1}\otimes M_{2m+1}\cong M_{2m+1} \oplus \Pi M_{2m+1}$$
 \end{sublemma}
\begin{proof}[Proof of Sublemma]
Let $\{a^{(k)}_j\}_{j=0}^k$ denote the basis of $M_{k}$ as above. 

We have a submodule of $M_{2m} \otimes M_{1}$, with the following basis (here and below the arrows indicate the action of $x$, the last vector being sent to zero by $x$):
\begin{align*}
&a^{(2m)}_{0}\otimes a^{(1)}_0 \mapsto a^{(2m)}_{1}\otimes a^{(1)}_0 + a^{(2m)}_{0}\otimes a^{(1)}_1 \mapsto  a^{(2m)}_{2}\otimes a^{(1)}_0 \mapsto \ldots \\
\mapsto &a^{(2m)}_{2j} \otimes a^{(1)}_0 \mapsto a^{(2m)}_{2j+1}\otimes a^{(1)}_0 + a^{(2m)}_{2j}\otimes a^{(1)}_1 \mapsto  a^{(2m)}_{2j+2}\otimes a^{(1)}_0 \mapsto \ldots \\ \mapsto &a^{(2m)}_{2m}\otimes a^{(1)}_0 \mapsto a^{(2m)}_{2m}\otimes a^{(1)}_1
\end{align*}
This submodule is in fact a direct summand of $M_{2m} \otimes M_{1}$, and is isomorphic to $M_{2m+1}$. Its complement summand is given by the submodule with the following basis:

\begin{align*}
&a^{(2m)}_{1}\otimes a^{(1)}_0 \mapsto -a^{(2m)}_{1}\otimes a^{(1)}_1 + a^{(2m)}_{2}\otimes a^{(1)}_0 \mapsto  a^{(2m)}_{3}\otimes a^{(1)}_0 \mapsto \ldots \\
\mapsto &a^{(2m)}_{2j+1} \otimes a^{(1)}_0 \mapsto -a^{(2m)}_{2j+1}\otimes a^{(1)}_1 + a^{(2m)}_{2j+2}\otimes a^{(1)}_0 \mapsto  a^{(2m)}_{2j+3}\otimes a^{(1)}_0 \mapsto \ldots \\ \mapsto &-a^{(2m)}_{2m-1}\otimes a^{(1)}_1 +a^{(2m)}_{2m}\otimes a^{(1)}_0
\end{align*}
The latter submodule is isomorphic to $\Pi M_{2m-1}$.

Thus the isomorphism between $M_{2m} \otimes M_{1}$ and $M_{2m+1} \oplus \Pi M_{2m-1}$ is given as follows: 

\begin{align*}
&\forall 0 \leq j \leq m, &a^{(2m)}_{2j}\otimes a^{(1)}_0 \mapsto (a^{(2m+1)}_{2j}, 0), \;\;\; &a^{(2m)}_{2j}\otimes a^{(1)}_1 \mapsto (a^{(2m+1)}_{2j+1}, -a^{(2m-1)}_{2j})\\
&\forall 0 \leq j \leq m-1,  &a^{(2m)}_{2j+1}\otimes a^{(1)}_0 \mapsto (0, a^{(2m-1)}_{2j}), \;\;\;
&a^{(2m)}_{2j+1}\otimes a^{(1)}_1 \mapsto (a^{(2m+1)}_{2j+2}, -a^{(2m-1)}_{2j+1})
\end{align*}

Similarly, the isomorphism between $M_{2m+1} \otimes M_{1}$ and $M_{2m+1} \oplus \Pi M_{2m+1}$ is given as follows: 

\begin{align*}
&\forall 0 \leq j \leq m, &a^{(2m+1)}_{2j}\otimes a^{(1)}_0 \mapsto (a^{(2m+1)}_{2j}, 0), \;\;\;&a^{(2m+1)}_{2j}\otimes a^{(1)}_1 \mapsto (a^{(2m+1)}_{2j+1}, -a^{(2m+1)}_{2j})\\
&\forall 0 \leq j \leq m-1, &a^{(2m+1)}_{2j+1}\otimes a^{(1)}_0 \mapsto (0, a^{(2m+1)}_{2j}), \;\;\;
&a^{(2m+1)}_{2j+1}\otimes a^{(1)}_1 \mapsto (a^{(2m+1)}_{2j+2}, -a^{(2m+1)}_{2j+1})\\
&{} &{} &a^{(2m+1)}_{2m+1}\otimes a^{(1)}_1 \mapsto (0, -a^{(2m+1)}_{2m+1})
\end{align*}

\comment{
we have the following submodule of $M_{2m+1} \otimes M_{1}$, with arrows indicating the action of $x$:
\begin{align*}
&a^{(2m+1)}_{0}\otimes a^{(1)}_0 \mapsto a^{(2m+1)}_{1}\otimes a^{(1)}_0 + a^{(2m+1)}_{0}\otimes a^{(1)}_1 \mapsto  a^{(2m+1)}_{2}\otimes a^{(1)}_0 \mapsto \ldots \\
\mapsto &a^{(2m+1)}_{2j} \otimes a^{(1)}_0 \mapsto a^{(2m+1)}_{2j+1}\otimes a^{(1)}_0 + a^{(2m+1)}_{2j}\otimes a^{(1)}_1 \mapsto  a^{(2m+1)}_{2j+2}\otimes a^{(1)}_0 \mapsto \ldots \\ \mapsto &a^{(2m+1)}_{2m+1}\otimes a^{(1)}_0 + a^{(2m+1)}_{2m}\otimes a^{(1)}_1
\end{align*}
This submodule is in fact a direct summand of $M_{2m+1} \otimes M_{1}$, and is isomorphic to $M_{2m+1}$. Its complement summand is given by the submodule

\begin{align*}
&a^{(2m+1)}_{1}\otimes a^{(1)}_0 \mapsto -a^{(2m+1)}_{1}\otimes a^{(1)}_1 + a^{(2m+1)}_{2}\otimes a^{(1)}_0 \mapsto  a^{(2m+1)}_{3}\otimes a^{(1)}_0 \mapsto \ldots \\
\mapsto &a^{(2m+1)}_{2j+1} \otimes a^{(1)}_0 \mapsto -a^{(2m+1)}_{2j+1}\otimes a^{(1)}_1 + a^{(2m+1)}_{2j+2}\otimes a^{(1)}_0 \mapsto  a^{(2m+1)}_{2j+3}\otimes a^{(1)}_0 \mapsto \ldots \\ \mapsto &a^{(2m+1)}_{2m+1}\otimes a^{(1)}_0 \mapsto -a^{(2m+1)}_{2m+1}\otimes a^{(1)}_1
\end{align*}
which is isomorphic to $\Pi M_{2m+1}$.
}
\end{proof}

The $\G^{(1|1)}$-module $M_{2k+1}\otimes M_{2m}$ is a direct summand of $$M_{2k} \otimes M_1 \otimes M_{2m} \cong M_{2k}  \otimes M_{2m} \otimes M_1 \cong \bigoplus_{s=\abs{k-m}}^{k+m} \Pi^{k+m-s} M_{2s} \otimes M_1 \cong \bigoplus_{s=\abs{k-m}}^{k+m} \Pi^{k+m-s} \left(M_{2s+1} \oplus \Pi M_{2s-1}\right)$$

Hence the indecomposable $\G^{(1|1)}$-summands in $M_{2k+1}\otimes M_{2m}$ necessarily have the form $M_{2s+1}$ or $\Pi M_{2s+1}$. Restricting to $\G$, we have:
 
\begin{align*}
 &Res^{\G^{(1|1)}}_{\G} \left( M_{2k+1}\otimes M_{2m} \right)\cong \left(V_k \oplus \Pi V_k\right) \otimes \left(V_m \oplus \Pi V_{m-1} \right)  \cong \\
 &\cong \bigoplus_{s=\abs{k-m}}^{k+m} V_s \oplus \bigoplus_{s=\abs{k-m+1}}^{k+m-1} V_s \oplus \Pi  \bigoplus_{s=\abs{k-m}}^{k+m} V_s \oplus \Pi  \bigoplus_{s=\abs{k-m+1}}^{k+m-1} V_s
 \end{align*}
 Due to the decomposition of $Res^{\G^{(1|1)}}_{\G} M_{2s+1}$, we conclude that 
 $$M_{2k+1} \otimes M_{2m}\cong \bigoplus_{s=min(\abs{k-m}, \abs{k-m+1})}^{k+m} \Pi^{k+m-s} M_{2s+1}$$

 Finally, consider the $\G^{(1|1)}$-module $M_{2k+1} \otimes M_{2m+1}$. Due to the Sublemma above, this is a direct summand of $M_{2k+1} \otimes M_{2m} \otimes M_1 $. By the previous computations, we have:
\begin{align*}
 M_{2k+1} \otimes M_{2m}\otimes M_1\cong \bigoplus_{s=min(\abs{k-m}, \abs{k-m+1})}^{k+m} \Pi^{k+m-s} M_{2s+1}\otimes M_1 \cong \\ \cong\bigoplus_{s=min(\abs{k-m}, \abs{k-m+1})}^{k+m} M_{2s+1} \oplus \Pi M_{2s+1}
 \end{align*}
 
 Hence $M_{2k+1} \otimes M_{2m+1}$ has only the summands of the form $M_{2s+1}$ or $\Pi M_{2s+1}$, each appearing with multiplicity at most $1$.
 Restricting to $\G$, we have:
 
\begin{align*}
 &Res^{\G^{(1|1)}}_{\G} \left( M_{2k+1}\otimes M_{2m+1} \right) \cong \left(V_k \oplus \Pi V_k\right) \otimes \left(V_m \oplus \Pi V_{m} \right)  \cong \\
 &\cong \bigoplus_{s=\abs{k-m}}^{k+m} V_s^{\oplus 2} \oplus \Pi  \bigoplus_{s=\abs{k-m}}^{k+m} V_s^{\oplus 2}
 \end{align*}
 
So we obtain: 
 $$M_{2k+1} \otimes M_{2m+1} \cong \bigoplus_{s=\abs{k-m}}^{k+m} M_{2s+1} \oplus \Pi M_{2s+1}$$
}
\end{proof}

\begin{corollary}
 The subcategory $\cU_{neat}$ is a Karoubian rigid SM subcategory of $\cU$.
\end{corollary}
\begin{proof}
 The indecomposable objects in $\cU_{neat}$ are precisely $M_{2m}$ for $m\geq 0$. By the computation above, the tensor product of any two such $\G^{1|1}$-modules lies again in $\cU_{neat}$. So $\cU_{neat}$ is closed under taking tensor products. The remaining claims are straightforward.
\end{proof}

\subsection{Semisimplification}

Consider the semisimplification of $\cU$. This is a $\fk$-linear monoidal functor $S: \cU \to \overline{\cU}$, where $\overline{\cU}$ is a semisimple tensor category.

Clearly, $S$ doesn't annihilate any object in $\cU_{neat}$.

Furthermore, we have:

\begin{lemma}
The composition $$S \circ S^*: \Rep(OSp(1|2)) \longrightarrow \overline{\cU}$$ is a tensor equivalence.
\end{lemma}

\begin{proof}
 The composition $S \circ S^*$ {as in the diagram below
$$\xymatrix{&{} &{} &\cU \ar[d]^{S} \\ &\Rep(OSp(1|2)) \ar[rru]^{S^*} \ar[rr]_-{S\circ S^*}  &{} &\overline{\cU} }$$
} is a $\fk$-linear SM functor between semisimple tensor categories. As such, it is automatically exact and faithful, and we only need to check that it is essentially surjective (which will make it automatically full). 
 
 Indeed, recall that $\{M_r\}_{r\geq 0}$ are the isomorphism classes of indecomposable objects in $ {\cU}$; for odd $r$, $\dim M_r =0$ so $S(M_r) = 0$, while for even $r$, $\dim M_r \neq 0$. So $$\{S(M_{2k}) = S\circ S^*(\widetilde{M}_{2k})\}_{k\geq 0}$$ are the isomorphism classes of simples in $ \overline{\cU}$, and $S \circ S^*$ is essentially surjective. 
\end{proof}

\subsection{Deligne filtration}\label{filtration}
Let $x$ be an odd nilpotent element acting on a finite-dimensional superspace $M$. Then $x$ defines a canonical finite increasing filtration\footnote{In the case of even $x$ this is the filtration which appears in the Hodge theory.}
$$\ldots\subset\mathcal F^{i}(M)\subset \mathcal F^{i+1}(M)\subset\ldots $$ satisfying the conditions
\begin{itemize}
  \item $x(\mathcal F^{i}(M))\subset\mathcal F^{i-2}(M)$;
  \item If $Gr^i(M):=\mathcal F^{i}(M)/\mathcal F^{i-1}(M)$ then $x^{{i}}: Gr^{i}(M)\to \Pi^iGr^{-i}(M) $ is an isomorphism for all $i\geq 0$.
  \end{itemize}
In particular, each object in $\cU$ is endowed with such a filtration, which is compatible with direct sums. On the indecomposable $\G^{(1|1)}$-module $M_{k}$, the filtration is given by: $$\mathcal{F}^{k-2i+1}(M_{k}) =\mathcal{F}^{k-2i}(M_{k}) = span\{a_j\}_{j\geq i} \; \text{ for } \; i\geq 0.$$

Choose the standard set of generators {$h, X, Y$ in  $\mathfrak{osp}(1|2)${ as in \ref{ssec:osp}}.

\begin{lemma}\label{lem:ospstr}
  For any $\G^{(1|1)}$-module $M$, $Gr^{ev}(M):=\bigoplus_{i\in\mathbb Z}Gr^{2i}(M)$ has a unique structure of $\mathfrak{osp}(1|2)$-module such that $h$ acts by grading and $X$ acts as $Gr(x)$.
\end{lemma}
\begin{proof} In order to define the $\mathfrak{osp}(1|2)$-module structure we have to define the action of $Y$ which satisfies the relations. We just write $M$ as a direct sum of modules of the form $M_{2k}$ and define this structure
  on each of $M_{2k}$ in the obvious manner.

  Now let us prove uniqueness.
  Suppose that there are two ways to define $Y,Y'$. Then $[h,Y-Y']=2(Y-Y')$ and $[X,Y-Y']=0$. That means that $Y-Y'\in\operatorname{End}_{\fk}(M)$ is the lowest weight vector of weight $2$ with {respect to the action of
  the Lie superalgebra} $\mathfrak{osp}(1|2)$ generated by $h, X, Y$. Since lowest weight can not be positive, we get $Y-Y'=0$.
\end{proof}

Let us denote by $T(M)$ the $\mathfrak{osp}(1|2)$-module associated to  $Gr^{ev}(M)$.

\begin{lemma}\label{lem:iso} $T$ defines a SM functor $\Rep(\G^{(1|1)})\to \Rep(OSp(1|2))$ isomorphic to $S$.
\end{lemma}
\begin{example}
We have: 
 $$\forall k \in \Z, \; T(M_{2k+1})=0 , \, T(M_{2k}) = \widetilde{M}_{2k} \; \text{ as vector spaces,}$$
\end{example}

\begin{proof} First, we check that $T$ is a functor. For this we consider a morphism of $\G^{(1|1)}$-modules $\alpha: M\to N$. It induces the morphism $Gr(\alpha):Gr^{ev}(M)\to Gr^{ev}(N)$. Note that $Gr(\alpha)$ commutes with
  action of $X=Gr(x)$ and $h$ and hence with the action of $Y$ by the same argument as in the proof of Lemma \ref{lem:ospstr}. This defines action of $T$ on morphisms and functoriality conditions are straightforward.

 Next, we show that $T$ is monoidal.

 Consider the filtration on $M \otimes M'$, and the subspace $\mathcal{F}^{2k} (M) \otimes \mathcal{F}^{2l}(M')$ for some $k, l$. To determine in which filtration  it sits, it is enough to consider this for indecomposable modules
 $M, M'$. Then for any $k, l$ we determine that $$\mathcal{F}^{2k} (M) \otimes \mathcal{F}^{2l}(M') \subset \mathcal{F}^{2(k+l)}(M\otimes M'). $$

 Now, under this embedding, we have: $$\mathcal{F}^{2k-1} (M) \otimes \mathcal{F}^{2l} (M') + \mathcal{F}^{2k} (M) \otimes \mathcal{F}^{2l-1} (M')  \subset \mathcal{F}^{2(k+l)-1}(M\otimes M') $$ Again,
 it is enough to check this statement for indecomposable modules $M, M'$, where this is a direct consequence of the computation of $\mathcal{F}^i (M_k)$ given above.

This gives us an embedding $Gr^{2k}(M) \otimes Gr^{2l}(M')  \to  Gr^{2k+2l}(M \otimes M') $. 

Hence we have a natural transformation
$$T(M) \otimes T(M')= \bigoplus_{k, l \in \Z} Gr^{2k}(M) \otimes Gr^{2l}(M')  \to \bigoplus_{i \in \Z} Gr^{2i}(M \otimes M') =  T(M \otimes M')  $$
which is an embedding for every $M, M'$.

To check that it is a (natural) isomorphism, one again needs to verify this only for indecomposable $M, M'$, where it is a direct computation. 

We conclude that $T$ is a ($\fk$-linear) monoidal functor. 

Clearly, $T$ is essentially surjective (since {the essential image of $T$ contains all the simple $OSp(1|2)$-modules $\widetilde{M}_{2k}$}) and thus full. Thus it is a full monoidal functor into a semisimple category $\Rep(OSp(1|2))$ and so factors through the functor $S$.
The claim now follows.

\end{proof}

\section{The super Jacobson-Morozov Lemma}\label{sec:JM}
\subsection{Definitions}
\begin{definition}
 Let $V$ be a vector superspace, and $x \in \End(V)$ an odd nilpotent operator. The element $x$ defines an action of $\G^{(1|1)}$ on $V$. 
 
 The element $x$ acts {\it neatly} in $V$ or  if as a $\G^{(1|1)}$-module $V$ decomposes into a direct sum of indecomposables $M_{2k}$ for some $k \in \Z_{\geq 0}$.
\end{definition}

In other words, all indecomposable $\G^{(1|1)}$-summands of $V$ have non-zero (super) dimension.

Let $G$ be an algebraic supergroup, and $\g = \mathrm{Lie}(G)$.
\begin{definition}

\mbox{}

\begin{enumerate}
 \item A nilpotent element $x\in\g_{\bar 1}$ is called {\it neat} if it acts 
neatly in \InnaC{every} finite-dimensional representation of $\g$. 
 \item By $\g_{neat}$ we denote the set of all neat nilpotent elements.
\end{enumerate}
 
\end{definition}

Any nilpotent element $x \in \g_{\bar 1}$ defines a homomorphism $i_x:\G^{(1|1)} \rightarrow G$ of algebraic supergroups, and vice versa. 

We will call such a homomorphism {\it neat} if $x \in \g_{neat}$. Clearly, if $x\in \g_{neat}$ and $x \neq 0$ then $i_x$ is injective. 

Let $R: \Rep(G) \longrightarrow \cU$ be the restriction functor with respect to the inclusion $i_x$. 
The fact that $i_x$ is neat means that $R(M) \in \cU_{neat}$ for any $M \in \Rep(G)$.

\begin{remark}
 The element $0 \in \g_{\bar 1}$ is always neat.
\end{remark}

\begin{example}
Let $G=GL(1|1)$. Then $\g = \gl(1|1) = \End^{\bullet} (\fk^{1|1})$ and $\g_{\bar 1}$ is spanned by $e:=\left(\begin{smallmatrix}
0 &1\\
0 &0                                                                         \end{smallmatrix} \right)$, $f:=\left(\begin{smallmatrix}
0 &0\\
1 &0                                                                         \end{smallmatrix} \right)$.

Then $e,f$ do not act neatly on the faithful $G$-representation $\fk^{1|1}$, so $e,f \notin \g_{neat}$. Thus $\g_{neat} = \{0\}$ in this case.

\end{example}
\begin{example}
Let $G=OSp(1|2)$. Then $\g_{\bar{1}} = \g_{neat}$.

\end{example}
\begin{example}\label{ex:neat_gl_1_2}
Let $G=GL(1|2)$. Then $\g =\gl(1|2) = \End^{\bullet} (\fk^{1|2})$ and the odd nilpotent cone is $$N_{\bar 1}= \left\{\begin{pmatrix}
0 &a &b\\
c &0 &0\\
d &0 &0                                                                        \end{pmatrix} \; : \;\; ac+bd =0\right\} $$
Now, 
$$\g_{neat}=\left\{\begin{pmatrix}0 &a &b\\
c &0 &0\\
d &0 &0                                                                        \end{pmatrix}\in N_{\bar 1} \; : (a,b)\neq (0,0), \,\, (c,d)\neq (0,0)\right\}\cup \{0\}.$$
This is proved by checking which elements of $\g_{\bar 1}$ act neatly on the faithful $G$-representation $\fk^{1|2}$, and using Lemma \ref{lem:neat_criterion}.
\end{example}

\subsection{Main statement}

\begin{theorem}[Super Jacobson-Morozov Lemma]\label{thrm:JM}
Let $G$ be a quasi-reductive algebraic supergroup. Let $i:\G^{(1|1)} \hookrightarrow G$ be a neat injective homomorphism. Then the inclusion $i$ can be extended to an injective homomorphism $\bar{i}: OSp(1|2) \hookrightarrow G$. This  extension is unique up to conjugation by an element of $G_{\bar 0}$.
\end{theorem}

\begin{proof}
  Let  $R: \Rep(G)\to \Rep(\G^{(1|1)})$ be the restriction functor associated with the inclusion $i:\G^{(1|1)} \hookrightarrow G$.

  Recall that the category $\Rep(G)$ has enough projective objects since $G$ is a quasi-reductive supergroup ({see Proposition \ref{prop:enough_proj}}). 

Now, since we assumed that $i$ is neat, we have: for every projective object $P \in \Rep(G)$, $R(P) \in \cU_{neat}$.

We will show that $$\Phi_i:=S \circ R: \Rep(G) \to \overline{\cU} \cong \Rep(OSp(1|2))$$ is an exact SM $\fk$-linear functor, hence inducing a homomorphism $\bar{i}: OSp(1|2) \hookrightarrow G$ \InnaC{by Remark \ref{rmk:conn_Tannakian}}. It will clearly be injective (since the supergroup $OSp(1|2)$ is simple). 

Since both $R, S$ are SM and $\fk$-linear (hence additive), so is the functor $\Phi_i$. So we only need to prove that $\Phi_i$ is exact.

First, notice that $\Phi_i(P) \neq 0$ for any projective $G$-module $P \neq 0$. Indeed, $R$ is faithful, and $S$ is faithful on the subcategory $\cU_{neat}$ to which $R(P)$ belongs.

Secondly, let $$ 0 \to M' \to M \to M'' \to 0$$ be a short exact sequence in $\Rep(G)$. Then for any projective $G$-module $P \neq 0$, $$ 0 \to P\otimes M' \to P \otimes M \to P \otimes M'' \to 0$$ is a split exact sequence of projective $G$-modules. Applying $\Phi_i$, we obtain a split exact sequence
$$ 0 \to \Phi_i(P)\otimes \Phi_i(M') \to \Phi_i(P) \otimes \Phi_i(M) \to \Phi_i(P) \otimes \Phi_i(M'') \to 0$$ in $\overline{\cU} \cong \Rep(OSp(1|2))$.

Now, $\overline{\cU} \cong \Rep(OSp(1|2))$ is a tensor category, so for any $X \in \overline{\cU}$, the endofunctor $X \otimes -$ is faithful and exact whenever $X \neq 0$.

Thus $\Phi_i(P) \otimes - $ is faithful and exact, so $$ 0 \to  \Phi_i(M') \to  \Phi_i(M) \to \Phi_i(M'') \to 0$$ is a short exact sequence in $\overline{\cU} \cong \Rep(OSp(1|2))$. This proves that $\Phi_i$ is exact.

Since $\Phi_i$ is an exact functor, it is isomorphic to the restriction functor $R_\varphi$ associated to some homomorphism (embedding) $\varphi:OSp(1|2)\to G$ \InnaC{(see Remark \ref{rmk:conn_Tannakian})}.

We fix  standard generators $h,{X, Y}$ in $\mathfrak{osp}(1|2)$ {as in Section \ref{ssec:osp}}, and consider the subgroup $\G^{(1|1)}$ with the Lie algebra generated by ${X}$. 

{Consider the following functors:
$$\xymatrix{&\Rep(G) \ar[rrr]^{R} \ar[rrrd]|{\;\; \Phi_i\;\; } \ar[rd]_{R_{\varphi}} &{} &{} &\cU \ar[d]^{S} \\ &{} &\Rep(OSp(1|2)) \ar[rr]_-{S\circ S^*}  &{} &\overline{\cU} }$$
}
Then we have an isomorphism of functors {$\Rep(G) \to \Rep(OSp(1|2))$}
$$S\circ S^*\circ R_\varphi\cong \Phi_{i},$$
where $S^*\circ R_\varphi$ is the restriction of $\varphi$ to our chosen $\G^{(1|1)}\subset \Rep(OSp(1|2))$. {This makes the diagram of functors above commutative}. 
On every $M\in\Rep(G)$ we have two Deligne filtrations $\mathcal F_{i(x)}(M)$ and $\mathcal F_{\varphi({X})} (M)$. 
\begin{lemma}\label{lem:aux1} There exists an automorphism $\psi$ of $\g$ such that the filtrations $\mathcal F_{i(x)}(\g)$ and $\mathcal F_{\psi\varphi({X})} (\g)$ coincide.
\end{lemma}

\begin{proof} Using Lemma \ref{lem:ospstr} we have $\Phi_i(\g)=\operatorname{Gr}_{\mathcal F_{i(x)}}\g$ and $S\circ S^*\circ{R_{\varphi}}(\g)=\operatorname{Gr}_{\mathcal F_{\varphi(X)}}\g$.
  Both these Lie superalgebras are isomorphic to $\g$, moreover they are isomorphic to each other as graded {Lie super}algebras. 
  This precisely means that the exists an automorphism
  $\psi$ of $\g$ such that
  $$\psi\mathcal F_{\varphi({X})}(\g)=\mathcal F_{\psi\varphi({X})}(\g)=\mathcal F_{i(x)}(\g).$$
\end{proof}

\comment{

}

\begin{lemma}\label{lem:aux2}
 Let $x_1, x_2 \in \g$ be two neat elements which define the same Deligne filtration $\mathcal F^{\bullet}(\g)$ on $\g$. Then $x_1, x_2$ are conjugate with respect to $G_{\bar{0}}$.
\end{lemma}
\begin{proof}
  First, let us note that since $x_1,x_2$ are neat, both must lie in $\mathcal F^{-2}(\g)$ but not in $\mathcal F^{-4}(\g)$.
  
  Let $G_f \subset G_{\bar 0}$ be the subgroup preserving the filtration $\mathcal F^{\bullet}(\g)$. 
  We have: $\mathrm{Lie}(G_f) \cong
  \mathcal F^{0}_{\bar 0}(\g)$ (the even part of $\mathcal F^{0}(\g)$). 
 
 Consider the subspace $Q$ of odd elements in $\mathcal F^{-2}(\g)$. Then $x_1, x_2 \in Q$. The group $G_f$ acts on $Q$. The orbits $O_1, O_2$ of $x_1, x_2$ under this action are open, since they have the same tangent space:
 $$T_{x_i} O_i \cong [Lie (G_f), x_i] = \mathcal F^{-2}_{\bar 1}(\g)$$
{(the last equality follows from the fact that $x_1, x_2$ define the same Deligne filtration $\mathcal F^{\bullet}(\g)$).}
 Since $O_1, O_2$ are Zariski open subsets of a vector space, they are dense. Hence they intersect and thus coincide. 
\end{proof}

Together, Lemmas \ref{lem:aux1}, \ref{lem:aux2} imply that $i(x)=\gamma\varphi({X})$ for {some} automorphism $\gamma$ of $\g$. Then $\bar i:=\gamma\varphi$ is the desired extension of $i$.

 Finally, let us show that $\bar i$ is unique up to conjugation in $G_{\bar 0}$. Indeed, let inclusions $\bar i_1,\bar i_2:\mathfrak{osp}(1|2)\to \g$ coincide on ${X \in \mathfrak{osp}(1|2)}$ and denote by $R_1,R_2$ corresponding restriction functors
 $\Rep(G)\to\Rep(OSp(1|2))$. We have an isomorphism of functors $S^*\circ R_1\cong S^*\circ R_2$ which after composing with $S$ produces an isomorphism $R_1\cong R_2$. By Tannakian formalism this implies that
 $\bar i_1$ and $\bar i_2$ are conjugate. 
\end{proof}

The following statement follows from the proof of  Theorem \ref{thrm:JM}. 

\begin{corollary}\label{lem:conj}
 Let $x\in \g_{\bar{1}}$ be a nilpotent element given by the embedding $i$ in Theorem  \ref{thrm:JM}. Then $x$ is $G_{\bar{0}}$-conjugate to $\varphi({X})$.
\end{corollary}

\begin{remark}
 The converse to Theorem \ref{thrm:JM} is also true: given a homomorphism $i: \G^{(1|1)} \to G$ which extends to a homomorphism $\bar{i}:OSp(1|2) \to G$, the homomorphism $i$ is neat. This follows from the fact that the  corresponding restriction functor $R: \Rep(G) \to \cU$ factors through the restriction functor $S^*: \Rep(OSp(1|2)) \to \cU$, so $R(M) \in \cU_{neat}$ for all $M \in \Rep(G)$. 
\end{remark}

\subsection{A corollary}

The following corollary of Theorem \ref{thrm:JM} for quasi-reductive supergroups $G$ may be considered as a generalization of the Kostant theorem, which states that there are finitely many
nilpotent orbits in the adjoint representation of a semisimple Lie algebra. 

\begin{proposition}\label{prop:orbits_neat}
  Let $G$ be a quasi-reductive supergroup. Then there are finitely many $G_{\bar 0}$-orbits of neat inclusions $i: \G^{(1|1)} \hookrightarrow G$.
  Equivalently,
  $\g_{neat}$ has finitely many $G_{\bar 0}$-orbits under the adjoint action.
\end{proposition}
\begin{remark}
 Equivalently, there exist only finitely many isomorphism classes of tensor functors $R: \Rep(G) \to \Rep(\G^{(1|1)})$ whose image lies in $\cU_{neat}$.
\end{remark}
\begin{remark}
{In general, the cone of odd nilpotent elements in $\g$ might have infinitely many orbits.}
\end{remark}

\begin{proof} Let $V$ be a faithful finite-dimensional $G$-module {(it exists by Remark \ref{rmk:faithful})}. Consider the inclusion $\g\subset\mathfrak{gl}(V)$. There are only finitely many $GL(V)_{\bar 0}$-orbits in
  $\mathfrak{gl}(V)_{neat}$ since there are finitely many non-equivalent representations of $\G^{(1|1)}$ in $V$. 
  
  Let $O$ be some $GL(V)_{\bar 0}$-orbit with non-trivial
  intersection with $\g_{\bar 1}$. We will prove that for any $x\in O\cap\g$ the tangent space $T_x(G_{\bar 0}x)$ coincides with $T_xO\cap\g$. 
  
  Indeed, {let $x \in O \cap \g_{\bar 1}$. By Lemma \ref{lem:neat_criterion}, we have: $x \in \g_{neat}$.}
  Consider the embedding
  $x\in\mathfrak{osp}(1|2)\subset \g$. Since $\Rep (\mathfrak{osp}(1|2))$ is semisimple we have the $\mathfrak{osp}(1|2)$-invariant decomposition $\mathfrak{gl}(V)=\g\oplus W$.
  That, in particular, implies $[x,W_{\bar 0}]\subset W_{\bar 1}$. Therefore we have
  $$T_xO\cap\g=[\mathfrak{gl}(V)_{\bar 0},x]\cap\g=([\g_{\bar 0},x]\oplus [W_{\bar 0},x])\cap \g=[\g_{\bar 0},x]=T_x(G_{\bar 0}x).$$
  That implies $\dim(O\cap\g)=\dim G_{\bar 0}x$ for any $x\in O\cap\g$. Therefore  $O\cap\g$ is a disjoint union of finitely many $G_{\bar 0}$-orbits. The statement follows.
  \end{proof}

\subsection{On the set \texorpdfstring{$\g_{neat}$}{g(neat)}}
  \begin{proposition}\label{prop:neatclassification} Let $\g$ be a Lie superalgebra such that $\g_{neat}=\g_{\bar 1}$. Then
    $$\g\cong \g'\oplus\mathfrak{osp}(1|2m_1)\oplus\dots\oplus\mathfrak{osp}(1|2m_k)$$
  for some $m_1,\dots,m_k\in\mathbb N$ and a Lie algebra $\g'$.
\end{proposition}
\begin{proof} We start with the following straightforward observations:
  \begin{enumerate}
  \item If $\g_{neat}=\g_{\bar 1}$ then $[x,x]\neq 0$ for any non-zero $x\in\g_{\bar 1}$.
    \item If $\g_{neat}=\g_{\bar 1}$ and $\mathfrak{h}$ is a quotient of $\g$, then $\mathfrak{h}_{neat}=\mathfrak{h}_{\bar 1}$.
    \end{enumerate}

    We are going to prove the statement by induction on $\dim\g_{\bar 0}+\dim\g_{\bar 1}$.
    We note that if $\g$ is simple then from Kac classification of simple superalgebras (1) holds only for $\g=\mathfrak{osp}(1|2m)$ or $\g_{\bar 1}=0$.

    Assume that $\g_{neat}=\g_{\bar 1}$.
    Let $\mathfrak m$ be some minimal non-zero ideal of $\g$.
    Then either $\mathfrak{m}$ is simple or abelian one-dimensional. Note also that by (2) in the latter case $\mathfrak m$ is even. If $\mathfrak{m}$ is simple then by above
    $\mathfrak m$ is either a Lie algebra or $\mathfrak{osp}(1|2m)$. First 
assume that $\mathfrak{m}\simeq \mathfrak{osp}(1|2m)$. 
\InnaC{Since $\mathfrak{osp}(1|2m)$ acts semisimply on its finite-dimensional 
modules, we have a splitting of $\mathfrak{m}$-modules $\g=\mathfrak{m}\oplus 
\mathfrak{l}$. This splitting is invariant under operators $ad(m), m\in 
\mathfrak{m}$, yet $\mathfrak{m}$ is an ideal; so $[m, t]=0$ 
for any $m\in \mathfrak{m}, t\in 
\mathfrak{l}$. Let $t_1,t_2 \in \mathfrak{l}$. Then for any $m \in 
\mathfrak{m}$, $[m,[t_1, t_2]]=0$ by Jacobi identity, and thus the 
projection of $[t_1, t_2]$ on $\mathfrak{m}$ sits in the center of 
$\mathfrak{m}$. Yet $\mathfrak{m}\simeq \mathfrak{osp}(1|2m)$ and it is 
center-less, so $[t_1, t_2] \in \mathfrak{l}$ and $\mathfrak{l}$ is an ideal 
in $\g$. Hence we have a decomposition of Lie superalgebras 
$\g=\mathfrak{m}\times \mathfrak{l}$} with $\dim\mathfrak{l}<\dim\g$ and the
    statement follows from the induction assumption.

    Now let us assume that $\mathfrak{m}$ is even. Consider $\mathfrak l:=\g/\mathfrak m$. Then by the induction assumption
    $$\mathfrak l\cong \mathfrak l'\oplus \mathfrak{osp}(1|2m_1)\oplus\dots\oplus\mathfrak{osp}(1|2m_k).$$ for some Lie algebra $\mathfrak l'$ and $m_1, \ldots, m_k \in \mathbb N$.
    Set $\mathfrak{s}:=\mathfrak{osp}(1|2m_1)\oplus\dots\oplus\mathfrak{osp}(1|2m_k)$ and $\mathfrak{n}:=\Ker(\g\to \mathfrak{s})$. Consider the exact sequence
    \begin{equation}\label{seq1}
      0\to\mathfrak n\to\g\to \mathfrak s\to 0.
      \end{equation}
Since the second cohomology of $\mathfrak s$ with coefficients in any module are zero, the above sequence splits. Furthermore, since $\mathfrak n$ is purely even the action of $\mathfrak s$ on $\mathfrak n$
is trivial. Therefore $\g\cong\mathfrak n\oplus\mathfrak s$ and the statement follows.
\end{proof}

\begin{proposition}\label{prop:neatgroup}[See also \cite{Wei}] Let $G$ be a 
connected algebraic supergroup with Lie superalgebra $\g$. Assume that 
\InnaC{every} $x\in\g_{\bar 1}$ acts neatly on \InnaC{every} representation of 
$G$.
  Then
    $$G\cong G'\times {OSp}(1|2m_1)\times\dots\times{OSp}(1|2m_k)$$
  for some $m_1,\dots,m_k\in\mathbb N$ and {an algebraic group $G'$}.
\end{proposition}
\begin{proof} The proof is similar to the proof of Proposition \ref{prop:neatclassification}. Consider a minimal connected non-trivial normal subgroup $M\subset G$. Since $\operatorname{Lie} M$ has no non-zero odd $x$
  such that $[x,x]=0$ we get that either $M\simeq OSp(1|2m)$ or $M$ is an algebraic group. In the former case we note that the decomposition $\g=\mathfrak{m}\oplus \mathfrak{l}$ can be lifted to
  $G=M\times L$ since $OSp(1|2m)$ does not have non-trivial discrete normal subgroups. In the latter case, we consider the sequence (\ref{seq1}) of Lie superalgebras and by the same reason it induces the decomposition
  for the groups $G=S\times N$ where $S={OSp}(1|2m_1)\times\dots\times{OSp}(1|2m_k)$ and $N$ is an algebraic group.
  \end{proof}

  \begin{corollary}\label{cor:neatgroup} Let $G$ be an algebraic supergroup with 
Lie superalgebra $\g$. Assume that \InnaC{every} $x\in\g_{\bar 1}$ acts neatly 
on \InnaC{every} representation of $G$. Then the (super)dimension of any 
indecomposable
    representation of $G$ is not zero.
  \end{corollary}
  \begin{proof} If $G$ is connected we use directly the proposition \ref{prop:neatgroup}. Any indecomposable representation of $G$ is of the form $M\otimes L_1\otimes\dots\otimes L_k$ where $L_i$ is an irreducible
    representation of $ {OSp}(1|2m_i)$ and $M$ is an indecomposable representation of $G'$; hence it has non-zero dimension. If $G$ is not connected, denote by $G_e$ the connected component of identity.
    {Consider any indecomposable representation $V$ of $G$ and its restriction $\Res^G_{G_e}(V)$ to $G_e$. Decomposing this into a direct sum of indecomposable finite-dimensional $G_e$- modules, we see that the finite group $G/G_e$ acts on the set of indecomposable direct summands. In other words, there exists an indecomposable $G_e$-module $V_e$ such that 
    $\Res_{G_e}(V)=\oplus_g V_e^{Ad_g}$, where $g$ runs over some subset of $G/G_e$, and $ V_e^{Ad_g}$ denotes the $G_e$-module $V_e$ with the action twisted by the automorphism $Ad_g$ of $G_e$. Therefore $\dim V$ is a multiple of $\dim V_e$ and hence not zero.}
     \end{proof}

     \begin{proposition}\label{prop:quasi_red_semisimple}
Let $G$ be quasi-reductive. Assume $\g_{neat}=\g_{\bar 1}$. Then $\Rep(G)$ is semisimple.  
\end{proposition}
\begin{remark}
 The condition $\g_{neat}=\g_{\bar 1}$ means that given a faithful 
representation $V$ of $G$, \InnaC{every} odd element $x \in \g_{\bar{1}}$ acts 
neatly. This is also equivalent to the condition that the image of {\it 
\InnaC{every}} tensor functor
  $R: \Rep(G) \to \Rep(\G^{(1|1)})$ lies in $\cU_{neat}$.
\end{remark}
\begin{proof} Follows from Corollary \ref{cor:neatgroup}. Indeed, if $P$ is projective indecomposable, then $\fk$ is a direct summand of $P\otimes P^*$ which implies that $\fk$ is projective.
\end{proof}

\section{The minuscule supergroup}\label{sec:minuscule}
\subsection{}
\VeraA{ Let $\mathfrak m$ be the $(1|1)$-dimensional Lie superalgebra with basis $\bar x, [\bar x,\bar x]$ 
such that $\bar x$ is odd. 
Consider the category $\Rep(\mathfrak m)$ of finite-dimensional $\mathfrak m$-modules. For any $\mathfrak m$-module $V$ we denote by $x$ the image of $\bar x$ in $\End_{\fk}(V)$.
\begin{lemma}\label{lem:m-indecomposable} Let $V\in \Rep(\mathfrak m)$ be indecomposable. Then either $x$ is nilpotent on $V$, or the map $x:V_{\bar 0}\to V_{\bar 1}$ is an isomorphism. 

In the former case, the action of $x$ on $V$ can be lifted to an indecomposable
  $\G^{(1|1)}$-module. 

  \InnaD{If $x$ is not nilpotent, then} there exist $\lambda\in\fk^*$ and bases
    $\{u_1,\dots,u_n\}$ of $V_{\bar 0}$
    and $\{v_1,\dots,v_n\}$ of $V_{\bar 1}$ such that
    $ x u_i=v_i$ and $xv_i=\InnaD{\lambda}u_i+u_{i-1}$ (we assume $u_0=v_0=0$).
  \end{lemma}
  \begin{proof} Let ${y}=[x,x]$ and ${y}={y}_s+{y}_n$ be the Chevalley-Jordan decomposition. Since $[y_s,\InnaD{y}]=[y_s,x]=0$ we obtain that $y_s$ acts on $V$ by some scalar $\lambda$. If $\lambda=0$ then $y_s=0$ and hence $x$ is nilpotent.
    If $\lambda\neq 0$ then $2x^2=y$ is an isomorphism. Hence $x:V_{\bar 0}\to V_{\bar 1}$ is an isomorphism. Therefore one can choose bases $\{u_1,\dots,u_n\}$ of $V_{\bar 0}$
    and $\{v_1,\dots,v_n\}$ of $V_{\bar 1}$ such that the matrix of $x$ in this basis has form $\left(\begin{matrix}0&C\\ 1_n&0\end{matrix}\right)$ for some $n\times n$-matrix $C$. Moreover, without loss of generality
  we may assume that $C$ is in Jordan normal form.
  Indecomposability of $V$ implies that $C$ has one Jordan block. This implies the last statement.
    \end{proof}

}   
By Lemma \ref{lem:splittannakian}, this category is equivalent to  $\Rep(\mathbb{M}) $
for some algebraic pro-supergroup $\mathbb M$. We call $\mathbb M$
the {\it  minuscule} supergroup.

\VeraA{\begin{lemma}
        The ring of functions on $\mathbb M$ is the supercommutative algebra
  $$\mathcal O(\mathbb M)\simeq \left(\fk[z_\lambda^{\pm 1}]_{\lambda\in \fk}/(z_\lambda^n-z_{n\lambda})_{\lambda\in\fk,n\in\mathbb Z}\right )\otimes\mathbb C[t,\omega],$$
    with $z_\lambda$ and $\InnaD{t}$ even and $\omega$ odd.
       \end{lemma}
 Note that $z_0=1$. We define the Hopf structure on $\mathcal O(\mathbb M)$ by coproduct
    $$\Delta(z_\lambda)=z_\lambda \otimes z_\lambda+\lambda(z_\lambda\omega\otimes z_\lambda\omega),\ \Delta(t)=t\otimes 1+1\otimes t+\omega\otimes\omega,\ \Delta(\omega)=\omega\otimes 1+1\otimes \omega,$$
    and antipode
    $$\sigma(z_\lambda)=z_\lambda^{-1},\ \sigma(t)=-t,\ \sigma(\omega)=-\omega.$$
    \begin{proof}
     Indeed, let $V$ be some finite-dimensional indecomposable $\mathfrak m$-module. If $\bar x$ acts nilpotently on $V$ then $V$ is an indecomposable representation of the quotient isomorphic to $\G^{(1|1)}$ where
    $\mathcal O(\G^{(1|1)})\subset\mathcal O(\mathbb M)$ is generated by $\omega$ and $t$. If $\bar x$ does not act nilpotently on $V$ we choose a basis as in Lemma \ref{lem:m-indecomposable} and define a right
    $\mathcal O(\mathbb M)$-comodule structure on $V$ by
    $$\rho(u_i)=u_i\otimes z_\lambda+u_{i-1}\otimes z_\lambda t+v_i\otimes z_\lambda\omega+v_{i-1}\otimes z_\lambda t \omega, $$
    $$\rho(v_i)=v_i\otimes z_\lambda +v_{i-1}\otimes z_\lambda t+\lambda u_i\otimes z_\lambda\omega+u_{i-1}\otimes z_\lambda(1+\lambda t)\omega.$$
  In other words, $V$ is a representation of the finite-dimensional quotient $G$ of $\mathbb M$. The corresponding Hopf subalgebra $\mathcal O(G)\subset\mathcal O(\mathbb M)$ is generated by $\omega,t,z_\lambda$.
    \end{proof}

    }
  
  \InnaC{
  \begin{lemma}\label{lem:M_connected}
   The supergroup $\mathbb M$ is connected.
  \end{lemma}
  \begin{proof}
  Let $V$ be an object in $Rep(\mathbb{M}) \cong Rep(\mathfrak{m})$. We denote by $\langle V\rangle$ the full subcategory of $Rep(\mathbb{M})$ generated by $V$ under taking finite direct sums and subquotients. 
  The statement that $\mathbb{M}$ is connected is equivalent to the statement that for every $V \in Rep(\mathfrak{m})$, if $\mathbb{M}$ (equivalently, $\mathfrak{m}$) does not act trivially on $V$ then $\langle V\rangle$ is not closed under $\otimes$ (see \cite[Corollary 2.22]{DM}).
  
  Indeed, let $V \in Rep(\mathfrak{m})$ on which $\mathfrak{m}$ does not act trivially. Consider the indecomposable direct summands of $V$. If $V$ has a direct summand $M$ on which $[\bar{x}, \bar{x}] \in \mathfrak{m}$ acts nilpotently but not trivially, then $ M^{\otimes s} \notin \langle V \rangle$ for $s>>0$ by Lemma \ref{lem:Clebsh_Gordan}, since $\langle V\rangle$ contains only finitely many non-isomorphic simple $\mathfrak{m}$-modules on which $[\bar{x}, \bar{x}]$ acts nilpotently. On the other hand, assume $V$ has a non-zero direct summand $M$ for which $[\bar{x}, \bar{x}] \rvert_M$ an isomorphism (see Lemma \ref{lem:m-indecomposable}). Let $\lambda \neq 0$ be the eigenvalue of $y_s$ (notation as in Lemma \ref{lem:m-indecomposable}) on $M$. Then $ M^{\otimes s}$, $s\in \Z_{\geq 0} $ will contain simple subquotients on which $[\bar{x}, \bar{x}]$ will have eigenvalues $s\lambda $, $s\in \Z_{\geq 0} $. So $ M^{\otimes s} \notin \langle V \rangle$ for $s>>0$, since $\langle V\rangle$ contains only finitely many non-isomorphic simple $\mathfrak{m}$-modules.
  \end{proof}}

\begin{lemma}\label{lem:minuscule} Let $G$ be an algebraic supergroup, $\g=\mathrm{Lie}(G)$ and $x\in\g_{\bar 1}$.
  There exists a unique homomorphism $i_x:\mathbb M\to G$ such that $Lie (i_x)(\bar x)=x$.
   \end{lemma}
   \begin{proof} The homomorphism $\mathfrak m\to\g$ of Lie superalgebras induces the restriction functor $\Rep(G)\to \Rep(\mathfrak m)$. 
     This functor defines a faithful SM functor $R_x:\Rep(G)\to \Rep(\mathbb M)$. The statement follows by Tannakian formalism, \InnaC{together with Lemma \ref{lem:M_connected} and Remark \ref{rmk:conn_Tannakian}}.
  \end{proof}
  \begin{remark} 
  The above Lemma implies that although the Lie superalgebra $\mathrm{Lie}(\mathbb M)$ is infinite-dimensional, its odd part is one-dimensional. Futhermore, the ``even'' part $\mathbb M_{\bar 0}$ is a direct product of $\G$ and the abelian
    reductive pro-group: the maximal pro-torus of $\mathbb M_{\bar 0}$.
    \end{remark}

\begin{remark}  The group $\mathbb{M}$ admits the supergroup $\G^{(0|1)}$ as the quotient by $\mathbb{M}_{\bar 0}$,
  the supergroup $\G^{(1|1)}$ as the quotient by the maximal pro-torus of $\mathbb M_{\bar 0}$. Finally, the quotient of $\mathbb M$ by $\G\subset \mathbb{M}_{\bar 0}$ gives a  ``superextension'' of  the maximal pro-torus.
  Every connected quasireducive supergroup with abelian even part and $1$-dimensional odd part is a quotient of this superextension. {An example of such a supergroup is $Q(1)$.}
\end{remark}

\begin{lemma}\label{lem:M-semisimplification}
 The semisimplification of the category $\Rep(\mathbb{M})$ is isomorphic to $\Rep(OSp(1|2))$.
\end{lemma}
\begin{proof}
We have a fully faithful SM $\fk$-linear functor $$\mathbf{R}: \cU \to \Rep(\mathbb{M})$$ corresponding to the quotient map $\mathbb{M} \to \G^{(1|1)}$. Let $S': \Rep(\mathbb{M}) \to \overline{\Rep(\mathbb{M})}$ denote the semisimplification functor of the category $\Rep(\mathbb{M})$. Then $S'\circ \mathbf{R}$ is a full SM $\fk$-linear functor from $\cU$ to the semisimple SM category $\overline{\Rep(\mathbb{M})}$. Such a functor necessarily factors through the semisimplification $$S: \cU \to \Rep(OSp(1|2))$$ and we obtain a full SM $\fk$-linear functor $\mathbf{R}': \Rep(OSp(1|2)) \to \overline{\Rep(\mathbb{M})}$ making the diagram below commutative 
$$\xymatrix{&{\cU} \ar[r]^-S \ar[d]_{\mathbf{R}} &{\Rep(OSp(1|2))} \ar[d]^{\mathbf{R}'} \\
&{\Rep(\mathbb{M})} \ar[r]^-{S'} &{\overline{\Rep(\mathbb{M})}} }$$
Now, $\Rep(OSp(1|2))$ is semisimple, so $\mathbf{R}'$ is automatically exact and hence faithful. It remains to check that it is essentially surjective, and then we can conclude that it is an equivalence. For this, it is enough to check that any simple object $\bar{M} \in \overline{\Rep(\mathbb{M})}$ lies in the essential image of $\mathbf{R}'$.

Indeed, recall that for any such $\bar{M}$ there exists an indecomposable object $M \in \Rep(\mathbb{M})$ such that $\bar{M} \cong S'(M)$ and $\dim M \neq 0$.
\VeraA{By Lemma \ref{lem:m-indecomposable} if $\dim M\neq 0$ then $\bar{x}\rvert_M$ is nilpotent.}

Then the action homomorphism $\mathbb{M} \to GL(M)$ factors through the quotient map $\mathbb{M} \to \G^{(1|1)}$, and hence $M \cong \mathbf{R}(\widetilde{M})$ for some $\widetilde{M} \in \cU$. This implies that $\bar{M} \cong S'\circ \mathbf{R}(\widetilde{M}) \cong \mathbf{R}' \circ S (\widetilde{M}) $, and so $\bar{M}$ lies in the essential image of $\mathbf{R}'$.

\end{proof}
As a corollary of the proof above, we have the following statement:
\begin{corollary}\label{cor:minuscule_neat}
 The full subcategory of $\Rep(\mathbb{M})$ of objects whose indecomposable summands have non-zero dimension is precisely $\cU_{neat}$, embedded in $\Rep(\mathbb{M})$ via the functor $\mathbf{R}$.
\end{corollary}

\section{General setting: tensor functors for odd elements}\label{sec:tensor_func}
\subsection{Definition}\label{ssec:gen_def_phi}
We now consider the most general setting.
Let $G$ be an algebraic supergroup with Lie \InnaC{super}algebra $\g$. Let 
$x\in \g_{\bar{1}}$.

Recall a homomorphism 
$i_x: \mathbb{M} \rightarrow G$ defined in Lemma \ref{lem:minuscule}.

Let $R_x: \Rep(G) \to \Rep(\mathbb{M})$ be the restriction functor with respect to $i_x$. Composing $R_x$ with the semisimplification functor $$S: \Rep(\mathbb{M}) \to \Rep(OSp(1|2))$$ we obtain a SM $\fk$-linear functor
$$\Phi_x:=S\circ R_x : \Rep(G) \longrightarrow \Rep(OSp(1|2)).$$ 

Let $\tilde{\g} := \Phi_x(\g)$. This is an $OSp(1|2)$-Lie algebra object.

The functor $\Phi_x$ is not necessarily exact on either side, but defines a SM $\fk$-linear functor 
$$\widetilde{\Phi_x}: \Rep(G) \longrightarrow \Rep_{OSp(1|2)}(\tilde{\g})$$ where the latter is the category of $OSp(1|2)$-equivariant representations of $\tilde{\g}$.

\begin{remark} It is not hard to see that if $\g$ is one of the classical superalgebras $\mathfrak{gl}(V),\ \mathfrak{osp}(V),\ \mathfrak{q}(V)$ or $\mathfrak{p}(V)$ then $\tilde \g$ is a classical superalgebra of the same type.
 Also in all examples we know if $\g$ is quasi-reductive then $\tilde \g$ is quasi-reductive but we do not know if it is true in general.
  \end{remark}
\subsection{Special case: nilpotent odd operator}\label{ssec:phi_nilp}
If $x$ is nilpotent, then $[x,x]$ is a nilpotent even element, and so the 
\InnaC{homomorphism} $\mathbb{M} \rightarrow G$ factors through the 
homomorphism $ \mathbb{M} \twoheadrightarrow \G^{(1|1)}$.
In that case, we will have a restriction functor $R:\Rep(G) \to \Rep(\G^{(1|1)})$ as in Section \ref{sec:JM}. Composing with the semisimplification functor $S:\Rep(\G^{(1|1)})\rightarrow \Rep(OSp(1|2))$, we obtain a $\fk$-linear SM  functor
$$\Phi_x:=S\circ R : \Rep(G) \longrightarrow \Rep(OSp(1|2)).$$
When $x$ is neat, this is precisely the functor considered in Theorem \ref{thrm:JM}, the restriction functor with respect to some embedding $OSp(1|2) \to G$.
\begin{remark}
 Let $x_1, x_2 \in \g_{neat}$. By Tannakian formalism, $\Phi_{x_1} \cong \Phi_{x_2}$ iff $x_1$, $x_2$ are conjugate with respect to $G_{\bar{0}}$.
\end{remark}

\subsection{Back to general case}
Let $x\in\g_{\bar 1}$ and $[x,x]=y_s+y_n$. Then $x$ acts nilpotently on $M^{y_s}$. Consider the Deligne filtration
$$\ldots\subset\mathcal F^{i}(M^{y_s})\subset \mathcal F^{i+1}(M^{y_s})\subset\ldots $$
as defined in Section \ref{filtration}. 
\begin{lemma}\label{lem:generalfunctor} $\Phi_x(M)\cong  T(M^{y_s})$ where $T=Gr^{ev}$ as defined in Lemma \ref{lem:iso}. 
\end{lemma}
\begin{proof} Note that $y_s$ acts semisimply on $M$ and any eigenspace of $y_s$ in $\mathbb M$-invariant. Every indecomposable $\mathbb M$-submodule which lies in the eigenspace with non-zero eigenvalue
  has superdimension zero as explained in the proof of Lemma \ref{lem:M-semisimplification}. Hence the statement follows from Lemma \ref{lem:iso}. 
\end{proof}

\subsection{Relation to the DS functor}\label{ssec:DS}

We now describe a special case of the construction above. 

Let $\G^{(0|1)}$ be the purely odd affine additive group, with $\mathrm{Lie}(\G^{(0|1)}) = \fk^{0|1}$ (a purely odd vector superspace with trivial bracket).

A representation of $\G^{(0|1)}$ is defined by a pair $(x, M)$ where $x: M \to \Pi M$ is an odd endormophism of $M$, such that $[x,x]=0$. 

Consider the group homomorphism $ \G^{(1|1)} \twoheadrightarrow \G^{(0|1)}$ and the corresponding (faithful, exact, $\fk$-linear SM) embedding

$$I:\Rep(\G^{(0|1)}) \longrightarrow \cU=\Rep(\G^{(1|1)}).$$

{For any $M \in \Rep(\G^{(0|1)})$}, the indecomposable summands of $I(M)$ in $\cU$ then have (categorical) dimensions either $\pm 1$ (the torsion part of $M$ seen as a module over the algebra of dual numbers $k[x]/x^2$) or $0$ (free part of $M$ over $k[x]/x^2$). These first summands are not annihilated by the semisimplification functor $S:\Rep(\G^{(1|1)})\rightarrow \Rep(OSp(1|2))$, and are sent to representations with trivial $OSp(1|2)$-action; the summands of second type are annihilated by the semisimplification functor $S$. Hence $$S \circ I: \Rep(\G^{(0|1)}) \longrightarrow \Rep(OSp(1|2))$$ is in fact given by a $\fk$-linear SM functor 
$$D:\Rep(\G^{(0|1)}) \longrightarrow \sVect$$ which sends $M$ to its homology $\Ker(x\rvert_M) / xM$.

Given an embedding $\G^{(0|1)} \hookrightarrow G$ into an algebraic supergroup $G$, consider $\g: =\mathrm{Lie}(G)$ as a Lie algebra object in $\Rep(G)$, and let $\tilde{\g}:=D(\g)$. 
The functor $D$ induces a $\fk$-linear SM functor $$\Rep(G) \longrightarrow \Rep(\tilde{\g})$$ which is precisely the Duflo-Serganova functor. Hence $\widetilde{\Phi_x}$ is a generalization of the Duflo-Serganova functor.

\begin{remark}
 In general the functor $\Phi_x$ does not satisfy the Hinich property (``exact in the middle''), satisfied by the Duflo-Serganova functors (see \cite[Lemma 30]{HPS}). For example, for $G=\G^{(1|1)}$ and $x \in \mathrm{Lie}(G)_{\bar 1} \setminus\{ 0\}$, the functor $\Phi_x$ is just the semisimplification functor $S$. Consider the short exact sequence of $\G^{(1|1)}$-modules:
 $$0 \to M_1 \to M_2 \to M_0 \to 0$$
 Then $S(M_1)=0$, $S(M_2) = \widetilde{M_2}$, $S(M_0) = \widetilde{M_0}$. This implies that the complex
 $$0 \to S(M_1) \to S(M_2) \to S(M_0) \to 0$$
 is not exact at the middle.
\end{remark}

\subsection{Support of a module}

Let $G$ be an algebraic supergroup, and $\g = \mathrm{Lie}(G)$.

\begin{definition}
 Let $M \in \Rep(G)$. We define the {\it support of $M$} as $$supp(M) := \{x \in \g_{\bar{1}} \rvert \; \Phi_x(M) \neq 0\}.$$
\end{definition}

\begin{remark} Note that by definition $supp(0)=\emptyset$ and $0 \in supp(M)$ for all non-zero $M \in \Rep(G)$.        
       \end{remark}

\begin{lemma}\label{lem:supp_obvious}
 Let $M, N \in \Rep(G)$. 
 \begin{enumerate}
  \item $supp(M \oplus N) = supp(M) \cup supp(N)$.
  \item $supp(M \otimes N) = supp(M) \cap supp(N)$.
 \end{enumerate}

\end{lemma}
\begin{proof}
The statements follow from the fact that the functor $\Phi_x$ is additive and monoidal.
\end{proof}

\begin{remark}
{The support of a module is not necessarily open nor closed. This can be seen from Proposition \ref{prop:proj_support} and the Example \ref{ex:neat_gl_1_2}, which shows that for $G=GL(1|2)$ the support of any projective module is $\g_{neat}$ and it is neither open nor closed.}
\end{remark}

Below we give some results about the relation between neat elements, support, and projective modules.

\begin{proposition}\label{prop:intersect_support}
$$\bigcap_{M \in \Rep(G), M\neq 0} supp(M) = \g_{neat}$$

\end{proposition}

\begin{proof}
First, let us show that for any $M \in \Rep(G)$, $\g_{neat} \subset supp(M)$. Indeed, let $x \in \g_{neat}$, let $i_x: \G^{(1|1)} \to G$ be the corresponding homomorphism and $R_x:\Rep(G) \to \cU$ the corresponding restriction. 

If $M \neq 0$, then $R_x(M)$ has at least one indecomposable summand, and since $x$ is neat, this summand is of non-zero dimension. Hence $\Phi_x(M) \neq 0$, and $x \in supp(M)$. This proves
$$ \g_{neat} \subset \bigcap_{M \in \Rep(G),M\neq 0} supp(M).$$
 
Next, we prove the inclusion in the other direction. 

Let $x \in \g_{\bar{1}}$. Let $i_x: \mathbb{M} \to G$ be the corresponding homomorphism and $R_x:\Rep(G) \to \Rep(\mathbb M)$ the corresponding restriction.

Assume that $x \in supp(M)$ for all non-zero $M \in \Rep(G)$; that is, $\Phi_x(M) \neq 0$ for all non-zero $M \in \Rep(G)$. 

By Lemma \ref{lem:not_neat_not_faith} below, this means that for all $M \in \Rep(G)$, the indecomposable summands of $R_x(M)$ have non-zero dimension. By Corollary \ref{cor:minuscule_neat}, this implies that
$$R_x(M) \in \cU_{neat} \subset \cU \subset \Rep(\mathbb{M}) \; \text{ for all } \; M \in \Rep(G).$$ The fact that $R_x(M) \in \cU$ for all $M$ implies that the embedding $i_x$ factors through the quotient map $\mathbb{M} \to \G^{(1|1)}$, and so $x$ is nilpotent. The fact that $R_x(M) \in \cU_{neat}$ for all $M$ implies that $x \in \g_{neat}$, as required.
\end{proof}

\begin{lemma}\label{lem:not_neat_not_faith}
 Let $x \in \g_{\bar{1}}$ and $M \in \Rep(G)$ be such that $R_x(M)$ has at least one indecomposable summand of dimension $0$ (so $x$ is not ``neat'' on $M$). Then the functor $$\Phi_x:\Rep(G) \to \Rep(OSp(1|2))$$ annihilates some non-zero module in $\Rep(G)$.
\end{lemma}
\begin{proof}
Consider the vector superspaces $M$ and $\Phi_x(M)$. The latter can be identified with a subspace of $R_x(M)$ which is the direct sum of all summands of $R_x(M)$ which have non-zero superdimension.

Since $\Phi_x$ is a SM functor, for every partition $\lambda$ we have: $$S^{\lambda}(M)  = 0 \; \text{ implies } \;S^{\lambda}(\Phi_x(M)) =0.$$

As vector spaces, the dimension of $M$ is greater than that of $\Phi_x(M)$. Hence these vector superspaces are not isomorphic. This implies (cf. \cite{D}, \cite{EHS}) that there exists a partition $\lambda$ such that $S^{\lambda}(M) \neq 0$ while $S^{\lambda}(\Phi_x(M)) =0$. Hence $\Phi_x(S^{\lambda}(M))=0$ while $S^{\lambda}(M) \neq 0$, as required.
\end{proof}

The following lemma is useful for determining the support of projective modules (Proposition \ref{prop:proj_support}).
\begin{lemma}\label{lem:annihilating_proj}
 Let $\T, \T'$ be two tensor categories, and $F: \T \to \T'$ be a SM $\fk$-linear functor (not necessarily exact). Assume that $F$ annihilates some object $M \neq 0$. Then $F$ annihilates all projective objects in $\T$.
 \end{lemma}
 \begin{proof}[Proof of \ref{lem:annihilating_proj}]
The map $ev: M\otimes M^* \to \triv$ is surjective in any tensor category, and $F(ev)=0$. 
Let $P$ be a projective object in $\T$. Then $\id_P \otimes ev: P\ \otimes M\otimes M^* \to P$ is surjective so it splits. 
Since $F(\id_P\otimes ev) =0$, we conclude that $F(P)=0$.
 \end{proof}
\begin{proposition}\label{prop:proj_support}
 For any non-zero projective module $P$ in $\Rep(G)$, we have: $supp(P) = \g_{neat}$. 
\end{proposition}
 \begin{proof}
By Proposition \ref{prop:intersect_support}, we have: $$ \g_{neat} \subset supp(P).$$
 
 In the other direction, let $x \in \g_{\bar{1}}$, and assume $x \notin \g_{neat}$. By Proposition \ref{prop:intersect_support}, there exists $M \in \Rep(G)$ such that $\Phi_x(M) =0$. Thus by Lemma \ref{lem:annihilating_proj}, $\Phi_x(P)=0$ for all projective $P$, so $x \notin supp(P)$. 
\end{proof}

We conjecture that for quasi-reductive supergroups the converse of the statement in Proposition \ref{prop:proj_support} also holds (cf. \cite{DS}):
\begin{conjecture}\label{conj:supp}
 Let $G$ be a quasi-reductive group, and $M \in \Rep(G)$. If $supp(M) = \g_{neat}$, then $M$ is projective.
\end{conjecture}

\begin{remark}
 The intersection of $supp(M)$ with the self-commuting cone $$\mathcal{N}_{comm}:=\{x \in \g_{\bar{1}}: [x,x] =0\}$$ has been studied extensively in \cite{DS} (there it is called the associated variety of $M$). 
 
 In \cite{SerVaintrob}, it has been shown there that for quasi-reductive supergroups $G$ of Kac--Moody type, we have: $$M\in \Rep(G) \; \text{ is projective iff } \; supp(M) \cap \mathcal{N}_{comm} = \{0\}.$$
In such cases, we have: $\g_{neat} \cap \mathcal{N}_{comm} =\{0\}$ by Proposition \ref{prop:proj_support}, and this implies Conjecture \ref{conj:supp}. However, it would be interesting to obtain this result in greater generality.
\end{remark}

Finally, in the lemma below, we give a convenient criterion to determine when $x \in \g_{neat}$.
\begin{lemma}\label{lem:neat_criterion}
 Let $G$ be a quasi-reductive algebraic supergroup, and $V$ be a faithful representation of $G$. Let $x\in \g_{\bar{1}}$, and assume $x\rvert_V$ is nilpotent and neat. 
 
 Then $x$ is nilpotent and $x \in \g_{neat}$.
\end{lemma}
\begin{proof}

Let $R_x: \Rep(G) \to \Rep(\mathbb M)$ be the restriction functor associated to $x$.

Let $\mathcal{A} \subset \Rep(G)$ be the full subcategory of (finite) direct sums of $G$-modules of the form $V^{\otimes r }\otimes (V^*)^{\otimes s}$, $r,s \in \mathbb{Z}_{\geq 0}$. Let $Kar(\mathcal{A})$ be the full subcategory of $\Rep(G)$ whose objects are direct summands of objects in $\mathcal{A}$. Since $V$ is faithful, any projective module sits in $Kar(\mathcal{A})$.

The full subcategory $\cU_{neat} \subset \Rep(\mathbb M)$ is closed under taking direct sums, direct summands and tensor products.

Since $x\rvert_V$ is nilpotent and neat, we have: $R_x(V) \in \cU_{neat} $, so $R_x(M) \in \cU_{neat}$ for all $M \in Kar(\mathcal{A})$ {(in particular, this implies that the group homomorphism $i_x:\mathbb{M}\to G$ factors through the homomorphism $\mathbb{M} \twoheadrightarrow \G^{(1|1)}$ and so $x$ is nilpotent)}.

Thus $\Phi_x(P) \neq 0$ for any projective module $P \in \Rep(G)$, {$P\neq 0$}. Since we assumed that $G$ is quasi-reductive, there exists at least one projective module $P\neq 0$ in $\Rep(G)$. Using Proposition \ref{prop:proj_support}, we conclude that $x \in supp(P)=\g_{neat}$.
\end{proof}

\section{Reductive envelopes}\label{sec:red_envelopes}

\subsection{}
Let $\mathcal{C}_{neat} \subset \Rep(G)$ be the full subcategory with objects 
$M$ so that \InnaC{every nilpotent} $x \in \g_{\bar{1}}$ is neat on $M$. This is a full 
Karoubi additive rigid SM subcategory.

Let $S:\mathcal{C}_{neat} \to \overline{\mathcal{C}_{neat}}$ be the semisimplification of $\mathcal{C}_{neat}$. 

\begin{proposition}\label{prop:red_envelopes}
 Let $G$ be an algebraic supergroup. Then $\overline{\mathcal{C}_{neat}} \cong 
\Rep(\overline{G})$ for some reductive\footnote{An algebraic supergroup $G$ is 
called {\it reductive} if $\Rep(G)$ is semisimple.} algebraic pro-supergroup 
$\overline{G}$. If \InnaC{every} indecomposable object in $\mathcal{C}_{neat}$ 
has non-zero dimension then we have a homomorphism $G \to \overline{G} \InnaC{\rtimes \mu_2}$. 
\end{proposition}
\begin{remark}
 \InnaC{If $G$ is connected, then in the setting of Proposition \ref{prop:red_envelopes} we obtain a homomorphism $G \to \overline{G}$ by Remark \ref{rmk:conn_Tannakian}.}
\end{remark}

\begin{example}\label{ex:red_env}
\begin{enumerate}

\mbox{}

 \item For a purely even supergroup $G=G_{\bar 0}$ the group $\overline{G}$ is the reductive envelope $G_{\bar 0}^{red}$ of the algebraic group $G_{\bar 0}$, \cite{AndreKahn}.

 \item For $G = \G^{(1|1)}$ or $G=\mathbb{M}$, we have: $\mathcal{C}_{neat} = \cU_{neat}$, and we obtain $\overline{G} \cong OSp(1|2)$.
 Note that for $G=\mathbb{M}$, the homomorphism $G \to \overline{G}$ is not injective.
\end{enumerate}

\end{example}

\begin{proof}
 First of all, notice that the category $\overline{\mathcal{C}_{neat}}$ is super-Tannakian. Indeed, by Deligne's theorem (see \cite{D}), every object in $\Rep(G)$ is annihilated by some Schur functor, and so the same holds for every object in $\mathcal{C}_{neat}$. The semisimplification functor $\mathcal{C}_{neat} \to \overline{\mathcal{C}_{neat}}$ is SM, so every object in $\overline{\mathcal{C}_{neat}}$ is annihilated by some Schur functor. Applying Deligne's theorem again, we conclude that $\overline{\mathcal{C}_{neat}}$ is super-Tannakian. Since $\Pi(\triv)$ is an object of $\overline{\mathcal{C}_{neat}}$, Lemma \ref{lem:splittannakian} implies $\overline{\mathcal{C}_{neat}} \cong \Rep(\overline{G})$ for some algebraic pro-supergroup $\overline{G}$.
 Furthermore, the category $\overline{\mathcal{C}_{neat}}$ is semisimple. Hence $\overline{G}$ is reductive.

 It remains to check that there is a section $S^*: \overline{\mathcal{C}_{neat}} \to \mathcal{C}_{neat}$, which will induce a homomorphism $G \to \overline{G}$. But this is a consequence of \cite[Corollary 3.11]{EO}.
 \end{proof}
 \begin{proposition}\label{prop:red_envelopes-quasi} If $G$ is quasi-reductive then any indecomposable $M$ in $\mathcal{C}_{neat}$ has non-zero dimension.
   \end{proposition}
\begin{proof} Take any indecomposable representation $V \in \mathcal{C}_{neat}$. 
 
 Let $K = \Ker(G \to GL(V))$, and let $G' = G/K$. The supergroup $G'$ is also quasi-reductive: its even part $G'_{\bar 0}$ is a quotient of the reductive group $G_{\bar 0}$, so it is also reductive.
 
 Now, $V$ has a natural structure of an indecomposable faithful representation of $G'$
 on which \InnaC{every} $x \in \mathrm{Lie}(G')_{\bar 1}$ is neat. By Lemma 
\ref{lem:neat_criterion}, we obtain: $\mathrm{Lie}(G')_{\bar 1} = 
\mathrm{Lie}(G')_{neat}$.
 
 Hence by Proposition \ref{prop:quasi_red_semisimple}, the category $\Rep(G')$ is semisimple. 
 
 This implies that $\dim V \neq 0$, since any non-zero indecomposable object in a semisimple tensor category has non-zero dimension.
\end{proof}

\begin{remark} If we drop the assumption that $G$ is quasi-reductive then it is not true in general that the dimension of any indecomposable object in $\mathcal C_{neat}$ is not $0$.
  For example, consider the supergroup $G=\G\times\G^{(1|1)}$ and $V=M_2\oplus \Pi M_2$ as a module over $\g_a^{(1|1)}$. Let $\{v_1,v_2,v_3\}$ and $\{w_1,w_2,w_3\}$ be bases of $M_2$ and $\Pi M_2$ respectively with action of odd generator ${X}$
  given by ${X}(v_i)=v_{i+1}$ and ${X}(w_i)=w_{i+1}$ for $i=1,2$ and ${X}(v_3)={X}(w_3)=0$. Define the action of {a generator} $y\in\mathrm{Lie}(\G)$ on $V$ by $$y(v_1)=w_2,\ y(v_2)=w_3,\ y(v_3)=y(w_1)=y(w_2)=y(w_3)=0.$$
  Then $V$ is an indecomposable $G$-module of superdimension $0$.
\end{remark}
\begin{remark}
  Let $G$ be an algebraic supergroup, and let $\widetilde{G}$ be the fundamental (super)group of the super-Tannakian category $\Rep(G)$. The algebraic supergroup $\widetilde{G} $ is isomorphic to the semidirect
  product $\mu_2\rtimes G$. If $\widetilde{G} \cong \mu_2 \times G$ then there is a homomorphism $\varepsilon:\mu_2\to G$ and we have a decomposition of abelian categories $$\Rep(G) \cong\Rep(G, \eps) \oplus \Pi \Rep(G, \eps).$$

  Consider the full Karoubi additive subcategory $\mathcal{C}'_{neat} := \mathcal{C}_{neat} \cap \Rep(G, \eps)$ of $\mathcal{C}_{neat}$. Taking its semisimplification $\overline{\mathcal{C}'_{neat}}$, the same argument as in the proof of Proposition \ref{prop:red_envelopes} shows that $\overline{\mathcal{C}'_{neat}} \cong \Rep(\overline{G}, \bar{\eps})$ for a reductive algebraic supergroup $\overline{G}$ and $\bar\eps: \mu_2 \to \overline{G}$. If the assumption of
  Proposition \ref{prop:red_envelopes} holds then there exists  a homomorphism $\phi:G \to \overline{G}$ such that $\phi \circ \eps = \bar{\eps}$.
  
  In particular, for $G = G_{\bar{0}}$ purely even and $\eps$ the trivial morphism, $\overline{G} \cong G^{red}_{\bar{0}}$ and $\overline{\mathcal C'_{neat}}$ is equivalent to $\Rep_{\Vect}(G^{red}_{\bar{0}})$.

  If $G = \G^{(1|1)}$ or $\mathbb{M}$, then $\widetilde G$ does not split into direct product of $G$ and $\mu_2$ and hence the abelian tensor category $\Rep(G)$ does not have splitting. However, the Karoubian category
  $\mathcal C_{neat}$ has a splitting
  $$\mathcal C_{neat}=\mathcal C_{neat}'\oplus \Pi\mathcal C_{neat}',$$
  where $\mathcal C'_{neat}$ is the Karoubian tensor subcategory generated by $\Pi M_{2}$. Then $\overline{\mathcal C'_{neat}}$ is equivalent to
$\Rep(OSp(1|2),\eps)$ where $\eps$ is the isomorphism of $\mu_2$ with the center of $OSp(1|2)_{\bar 0}$.
\end{remark}

\end{document}